%% file: note.tex
\documentclass[10pt,notitlepage]{amsart}
\input{macros}

\newcommand{\ps}[1]{[\![#1]\!]}

\renewcommand{\c}{\mathbf{c}}
\newcommand{\ch}{\mathrm{ch}}

\newcommand{\CH}{\mathrm{CH}}
\newcommand{\Sm}{\mathrm{Sm}}
\newcommand{\Sch}{\mathrm{Sch}}

\newcommand{\MU}{\mathrm{MU}}
\newcommand{\MGL}{\mathrm{MGL}}

\renewcommand{\P}{\mathscr{P}}

\newcommand{\cN}{\mathcal{cN}}
\newcommand{\N}{\mathcal{N}}

\newcommand{\bB}{\mathbf{B}}
\newcommand{\bN}{\mathbf{N}}
\newcommand{\bT}{\mathbf{T}}

\newcommand{\bG}{\mathbf{G}}
\newcommand{\bH}{\mathbf{H}}
\newcommand{\bGLr}{\mathbf{\GL_r}}

\DeclareMathOperator{\ind}{ind}

\DeclareMathOperator{\Cycle}{Cycle}

\newcommand{\omegaBT}{\omega_{*,\bT}}

\newcommand{\omegaBG}{\omega_{*,\bG}}
\newcommand{\MUBG}{\MU_*(B\bG(\CC))}

\newcommand{\otg}{[\tfrac{1}{\tau_\bG}]}

\begin{document}
\title{Algebraic cobordism of varieties with $\bG$-bundles}
\author{Anatoly Preygel\\\today}
\maketitle
\begin{abstract}Lee and Pandharipande studied a ``double point'' algebraic cobordism theory of varieties equipped with vector bundles, and speculated that some features of that story might extend to the case of varieties with principal $\bG$-bundles.  This note shows that this expectation holds rationally, and more generally after inverting the torsion index of the group, for reductive $\bG$.  We show that (after inverting the torsion index) the full theory for bundles on varieties is an extension of scalars of standard algebraic cobordism, that the theory for a point is dual to $\Omega^*(B\bG)$, and describe how the theory for $\bG$ compares to that for a maximal torus $\bT$.
\end{abstract}

\section{Introduction}  
The central role of complex cobordism $MU$ in algebraic topology is partially due to its multiple incarnations.  The following history of \emph{algebraic} cobordism reflects this:
\begin{itemize} \item Voevodsky defined bi-graded \emph{motivic cobordism} groups $\MGL^{*,*}(X)$ for smooth $k$-schemes $X$ as the (motivic) representable cohomology theory associated to the Thom ($\PP^1$-)spectrum $\MGL$.  Just as the $(2n,n)$ piece of motivic cohomology admits a geometric description given by the classical Chow groups $\CH^n(X) \isom H^{2n,n}(X)$, so it is for motivic cobordism:
  \item In \cite{LM}, Levine and Morel constructed \emph{algebraic cobordism} $\Omega^*(X)$ as a universal ``oriented'' cohomology sense in the style of Quillen's treatment of complex cobordism via formal groups laws.  Conditional on unpublished work of Hopkins-Morel, Levine proved $\Omega^n(X) \isom \MGL^{2n,n}(X)$ so that this did in fact capture geometrically the appropriate piece of motivic cobordism.
  \item In \cite{LP1}, motivated by applications to Donaldson-Thomas theory, Levine and Pandharipande were able to give a purely geometric presentation, in the spirit of the construction from which complex co\emph{bordism} draws its name. This construction of $\omega_*(X)$ removed the need to algebraically insert the formal group law satisfied by $c_1$ and relied instead on geometric ``double point'' relations. These relations geometrically encode the formal group law in the simplest model for an equality $[C]=[A]+[B]$ of divisor classes: a smooth divisor $C$ rationally equivalent to the union of smooth divisors $A$ and $B$ intersecting transversely.
\end{itemize}

In \cite{LP}, also motivated by enumerative applications, Lee and Pandharipande defined an analogous ``double point'' cobordism theory for vector bundles, and for lists of line bundles, on varieties. They computed ``$\omega_{*,\bGLr}(X)$'' in terms of $\omega_*(X)$ via the following Kunneth-type theorem:
\begin{theorem}[Lee-Pandharipande]\label{thm:LP}  Let $\bG$ be one of $\mathbf{GL_{r}}$ or $\GG_m^r$ and denote by $\omega_{*,\bG}(X)$ the double point cobordism theory of $\bG$ bundles (see $\S$\ref{ss:algcob} for a precise definition.)   Set $\omega_{*,\bG} = \omega_{*,\bG}(\pt)$ and $\LL = \omega_{*}(\pt)$. Then, 
  \begin{enumerate}
    \item The natural map
      \[\gamma^X_{\bG}: \omega_*(X) \otimes_{\LL} \omega_{*,\bG} \longrightarrow \omega_{*,\bG}(X) \] 
      \[ [\pi: Y \to X] \otimes [Y', \P] \longmapsto [\pi \circ p_1: Y \times Y' \to X, p_2^* \P] \] 
      is an isomorphism of $\LL$-modules
    \item Suppose (for convenience) that $k \subset \CC$.  Then, the natural map (see Construction~\ref{constr:vartheta})
      \[ \vartheta: \omega_{*,\bG} \to \MUBG \] is an isomorphism of $\LL$-modules. (Recall that the natural map $\LL \to \MU_*(\pt)$ is known to be an isomorphism.)
    \item Suppose $\bG = \mathbf{GL_{r}}$ and set $\bT = \GG_m^r \subset \bG$.  Then, the natural map $\omega_{*,\bT}/S_r \to \omega_{*,\bG}$ is an isomorphism of $\LL$-modules.
  \end{enumerate}
\end{theorem}

One can easily define a ``double point'' cobordism theory $\omega_{*,\bG}(-)$ for more general principal $\bG$-bundles on varieties.  The authors of \cite{LP} were thus led to the following speculation: Is $\omega_{*,\bG}(\pt)$ dual to $\MU^*(B\bG)$  (or rather, the closely related $\Omega^*(B\bG)$) for classical groups $\bG$?  More generally, one can ask to what degree Theorem~\ref{thm:LP} remains true for more general groups $\bG$.

The goal of this note is to show how the methods and results of \cite{LP} extend to imply an analog of Theorem~\ref{thm:LP} \emph{rationally} for \emph{any} reductive algebraic group $\bG$.  One can also give a precise description of what happens \emph{integrally} for \emph{special} groups like $\mathbf{SL_n}$ and $\mathbf{SP_n}$ (or more generally, after inverting certain \emph{torsion primes}).

In addition to the abstract statements above, \cite{LP} obtained explicit geometric bases for $\omega_{*, \GG_m^r}$ and $\omega_{*,\mathbf{GL_r}}$.  Unfortunately, this step does not immediately generalize for a simple algebraic reason which is already visible in comparing the situations in ordinary cohomology of $\mathbf{GL_n}$ and $\mathbf{SL_n}$.  The formal structure is similar in both situations: If $\bT \subset \bG$ is the maximal torus, $W$ the Weyl group, and $\oh{\bT}$ the character lattice of $\bT$, then $H^*(B\bT) \isom \Sym \oh{\bT}$ and $H^*(B\bG) \isom (\Sym \oh{\bT})^W$ in these cases.  But the nature of the Weyl action is different: For $\mathbf{GL_n}$, if we fix a basis $\lambda_1,\ldots,\lambda_n$ for $\oh{\bT}$, then $W$ acts on this basis as a set.  For $\mathbf{SL_n}$, if we fix a basis $\lambda_1,\ldots,\lambda_{n-1}$ for $\oh{\bT}$, then some elements $\sigma \in W$ will act as, e.g., $\sigma \lambda_1 = -\lambda_1$.


\subsection{Rational results}
\begin{na} We work over an algebraically closed field $k$ of characteristic zero.
\end{na}

\begin{theorem}[Main Theorem]\label{thm:main}  Let $\bG$ be a reductive algebraic group and denote by $\omega_{*,\bG}(X)$ the double point cobordism theory of $\bG$ bundles (see $\S$\ref{ss:algcob} for a precise definition).  Set $\omega_{*,\bG} = \omega_{*,\bG}(\pt)$ and $\LL = \omega_{*}(\pt)$. Then, 
  \begin{enumerate}
    \item The natural map
      \[ \gamma^X_{\bG}: \omega_*(X) \otimes_{\LL} \omega_{*,\bG} \longrightarrow \omega_{*,\bG}(X) \] 
      \[ [\pi: Y \to X] \otimes [Y', \P] \longmapsto [\pi \circ p_1: Y \times Y' \to X, p_2^* \P] \] 
      induces an isomorphism $(\gamma^X_{\bG})_{\QQ}$ of $\LL_{\QQ}$-modules upon tensoring $- \otimes_{\ZZ} \QQ$.
    \item If $k \subset \CC$, natural map  (see Construction~\ref{constr:vartheta})
      \[ \vartheta_\bG: \omega_{*,\bG} \longrightarrow \MUBG \] induces an isomorphism $\vartheta_{\QQ}$ of $\LL_{\QQ}$-modules upon tensoring $- \otimes_{\ZZ} \QQ$.
  \end{enumerate}
\end{theorem}

The proof in \cite{LP} proceeds by first proving getting control over the case for tori and then deducing the case of $\mathbf{GL_r}$ from this by studying the torus of diagonal matrices together its conjugation action by permutation matrices. Our proof will proceed along similar lines, by establishing the following analog:
\begin{theorem}\label{thm:coinv} Suppose $\bG$ is a reductive group with maximal torus $\bT$ and Weyl group $W$.  Let $X \in \Sch_k$.  Then, the natural map
  \[ \ind_\bT^\bG: \omegaBT(X)/W \longrightarrow \omegaBG(X) \] 
  \[ [\pi: Y \to X, \P] \longmapsto [\pi: Y \to X, (\P \times \bG)/\bT)] \] induces an isomorphism $(\ind_\bT^\bG)_{\QQ}$ of $\LL_\QQ$-modules upon tensoring $- \otimes_\ZZ \QQ$. 
\end{theorem}

\subsection{Better-than-rational results} 
Each compact Lie group (or reductive algebraic group) has certain \emph{torsion primes}. In the compact Lie case, these are the prime numbers $p$ such that $H^*(BG, \ZZ)$ contains non-trivial $p$-torsion.  Upon inverting the \emph{torsion index} $\tau_{\bG}$, a number divisible by precisely these primes, we obtain that $H^*(BG, \ZZ\otg)$ is a free $\ZZ$-module and with a bit more work that $\MU_*(BG) \otimes_\ZZ \ZZ\otg$ is a free $\LL\otg$-module.  The following Theorems tell us that the structure of $\omega_{*,\bG}(X)$ is easy to describe \emph{away from these interesting primes}.

\begin{theorem}\label{thm:better-than-Q} Suppose $\bG$ is a reductive algebraic group with torsion index $\tau_{\bG}$ (see $\S$\ref{ss:torsion} for details).   Let $\omega^{Zar}_{*,\bG}(X)$ denote the double-point cobordism theory of \emph{Zariski-locally trivial} principal $\bG$-bundles.  Then,
  \begin{enumerate}
    \item The natural map $\omega^{Zar}_{*,\bG}(X) \to \omega_{*,\bG}(X)$ induces an isomorphism of $\LL\otg$-modules after inverting $\tau_{\bG}$.
    \item The map $\gamma^X_{\bG}: \omega_*(X) \otimes_{\LL} \omega_{*,\bG} \to \omega_{*,\bG}(X)$ of Theorem~\ref{thm:main} induces an isomorphism $\gamma^X_{\bG}\otg$ of $\LL\otg$-modules.
    \item $\omegaBG\otg$ is a projective $\LL\otg$-module, and its dual is $\Omega^*(B\bG)\otg \isom \Omega^*(B\bT)\otg^W = \LL\otg\ps{\oh{\bT}}^W \isom \LL\otg\ps{\lambda_1,\ldots,\lambda_r}^W$ ($\isom \MU^*(B\bT(\CC), \ZZ\otg)^W$ if we're over $\CC$) where $r$ is the rank of $\bG$ (see $\S$\ref{ss:char-notation} for details on the notation).
  \end{enumerate}
\end{theorem}

As a corollary, we obtain:
\begin{theorem}\label{thm:special} Suppose $\bG$ is a \emph{special} reductive algebraic group, in the sense that every principal $\bG$-bundle is Zariski-locally trivial (e.g., iterated extensions of $\GG_m$, $\mathbf{SL_n}$, and $\mathbf{Sp_n}$).  Then,
  \begin{enumerate}
    \item The map $\gamma^X_{\bG}: \omega_*(X) \otimes_{\LL} \omega_{*,\bG} \to \omega_{*,\bG}(X)$ of Theorem~\ref{thm:main} is an isomorphism of $\LL$-modules.
    \item $\omegaBG$ is a projective $\LL$-module, and its dual is $\Omega^*(B\bG) \isom \Omega^*(B\bT)^W = \LL\ps{\oh{\bT}}^W \isom \LL\ps{\lambda_1,\ldots,\lambda_r}^W$ ($\isom \MU^*(B\bT(\CC))^W$ if we're over $\CC$) where $r$ is the rank of $\bG$ (see $\S$\ref{ss:char-notation} for details on the notation).
  \end{enumerate}
\end{theorem}


\begin{remark} The identification of the dual module in Theorem~\ref{thm:better-than-Q}(ii) is via ``characteristic operations.'' They are an algebraic model of the following: $MU$ is a ring spectrum, so that $MU_*(B\bG(\CC))$ is a module over the ring $MU^*(B\bG(\CC))$. 
\end{remark}

It is also reasonable to ask what we can say about $\omega_{*,\bG}$, rather than merely its dual.  This is somewhat addressed by the following Proposition.
\begin{prop}\label{prop:direct-summand} Suppose $\bG$ and $\tau_{\bG}$ are as in Theorem~\ref{thm:better-than-Q}.  Then, \begin{enumerate} \item $\omega_{*,\bG}\otg$ is a direct summand of the free $\LL\otg$-module $\omega_{*, \bT}\otg$.  Thus, it is the quotient of $\omega_{*, \bT}\otg$ by its complementary direct summand, which is precisely the submodule of $x \in \omega_{*, \bT}\otg$ which pair trivially with all classes in $\LL\otg\ps{\oh{\bT}}^W \subset \LL\otg\ps{\oh{\bT}} \isom \MU^*(B\bT, \ZZ\otg)$ under $\ip{-}{-}_\bT$ (see Section~\ref{sec:operations}).
  \item $\omega_{*,\bG}\otg \isom (\omega_{*,\bT}\otg/W)/\text{torsion}$. \end{enumerate}
\end{prop}

\subsection{Relation to nearby definitions} At least if $X$ is projective, $\omega_{*,\bG}(X)$ can be thought of as a geometric version of ``bordism of free $\bG$-spaces over $X$,'' which one can regard as (at least the free piece of) a $\bG$-equivariant variant of algebraic bordism.\footnote{It is not hard to make a reasonable analogous definition of $\omega_{*,\bG}(X)$ for $X$ a scheme with $\bG$-action, satisfying some of the usual compatibilities such as $\omega_{*,\bG}((X \times \bG)/\bH) = \omega_{*,\bH}(X)$ for a $\bH$-scheme $X$, $\omega_{*,\bG}(X) = \omega_{*}(X/\bG)$ if the action is free, etc.  However, various other formal properties (e.g., localization, pullbacks) seem to be tricky at best.}  There is a more standard definition of  (the free piece of) ``equivariant algebraic cobordism'' $\Omega^*_\bG(X)$, worked out in \cite{Deshpande}, as (up to details of the limiting argument) $\Omega^*((X \times E\bG)/\bG)$.  These latter objects are sometimes reindexed and relabelled as $\Omega_*^\bG(X)$, but this notation is perhaps misleading since the limiting is still carried out using pullbacks rather than pushforwards.  Concretely, take $X = \pt$: $\Omega_*^\bG(\pt) = \Omega^{-*}(B\bG)$ lives in all degrees ($\LL$ is pushes the degree up, while the images of elements in Chow live in non-positive degree), completely unlike what one might expect of $\MU^G_*(\pt) = \MU_*(BG)$ in topology.  There is however a relationship between $\omega_{*,\bG}(X)$ and $\Omega^*_\bG(X)$: $\omega_{*,\bG}(X)$ is a module for $\Omega^*_\bG(\pt)$, and this is precisely the structure of ``characteristic operations'' that we will use.

\subsection{Plan of the paper} Section~\ref{sec:bg} reviews some background on torsion indices and algebraic cobordism.  
Section~\ref{sec:charBG} develops what we need of characteristic classes in algebraic cobordism.
Section~\ref{sec:operations} develops some operations on double-point cobordism ``with extra structure.''
The remaining three sections are devoted to proving the main results.  Section~\ref{sec:bTbB} provides the main new input (Lemma~\ref{lem:bTbB}) and recalls a slight reformulation of results of \cite{LP}.
Section~\ref{sec:rational} gives a proof of the rational statements that is just an elaboration of the proof given in \cite{LP} for $\mathbf{GL_r}$, with Lemma~\ref{lem:bTbB} and rational characteristic classes for $\bG$-bundles providing ``drop-in'' replacements for parts of the arguments there.

Section~\ref{sec:better-than-Q} contains the proofs of the ``better-than-rationally' results, which involve some algebraic trickery (Prop.~\ref{prop:retract}) to make up for the fact that $\omega_{*,\bT}/W$ need not be free.  We emphasize that the proofs are relatively formal from Lemma~\ref{lem:bTbB}, the properties of torsion indices, and the existence of enough characteristic classes.  It should be noted that this section is logically independent of the previous one.\footnote{Making this literally true requires minor rewording. Remark~\ref{rmk:depends} indicates an alternate path to the better-than-rationally results which depends on Section~\ref{sec:rational} but does not require the construction of better-than-rational characteristic classes.  However, we thought it worthwhile to include Section~\ref{sec:rational} for the slightly more hands-on feel of its proofs and to include better-than-rational characteristic classes for their own sake.}

\subsection{Acknowledgements} The author thanks Davesh Maulik for bringing this problem to his attention, and for encouragement while this work was being carried out and written.  The author thanks Rahul Pandharipande for helpful conversation, and for his interest in this work.  This work was carried out while the author was supported by an NSF Graduate Fellowship.


\subsection{Notation} 
We work throughout over a fixed \emph{algebraically closed field $k$ of characteristic zero}.  (Algebraically closed can be eliminated by e.g., requiring tori to be split, or working only rationally.  The characteristic zero assumption seems deeply ingrained for the time being.)  It will sometimes be convenient (for comparing to topology) to assume that our field embeds in (or just is) $\CC$, and we will do so without further notice.

Unless otherwise stated, we adopt the following notation
\begin{itemize}
  \item $\Sch_k$, $\Sm_k$: Category of separated schemes of finite type over $k$, the full subcategory of smooth quasi-projective $k$-schemes.
  \item $\bH$; $\bG$, $\bG_0$, $\bB$, $\bT$, $\bN$, and $W$: A linear algebraic group; a reductive algebraic group, its connected component of identity, a Borel subgroup, a maximal torus in $\bB$, the unipotent radical of $\bB$, and the Weyl group $W = N_{\bG}(\bT)/\bT$.
  \item $G$, $T$: Compact real forms of $\bG(\CC)$, $\bT(\CC)$.
  \item $\P$, $X_{\P}$, $p: (\bG/\bB)_{\P} \to Y$, $\P^{\bB} \to (\bG/\bB)_{\P}$, $\P^{\bT}$: $\P$ is a principal $\bG$-bundle on $Y$. If $X$ is a scheme acted on by $\bG$ then $X_{\P} \eqdef (\P \times X)/G$ is the \emph{associated space}.  (If the action of $\bG$ on $X$ is polarized by some ample line bundle on $X$, then $X_{\P}$ will be a $Y$-scheme.  This will be the case in all cases of interest to us.)  The generalized flag bundle $p:\P/\bB = (\bG/\bB)_{\P} \to Y$ of $\P$ carries a natural $\bB$-bundle denoted $\P^{\bB}$: this is just $\P \to \P/\bB$ regarded as a $\bB$-bundle on the flag variety, and it is a $\bB$-reduction of $p^* \P$.  If $\P$ is a principal $\bB$-bundle (in most cases, $\P^{\bB}$), then $\P^{\bT} = (\bB/\bN)_{\P} = \P/\bN$ is the associated principal $\bT$-bundle.
  \item $\LL$, $\mm = \LL_{\geq 1}$: The Lazard ring (the ring over which the universal formal group law is defined, and $\omega_*(\Spec k)$), and the ideal of positively graded elements in $\LL$ (=the augmentation ideal of $\LL \to \ZZ = \LL/\mm$).
  \item $\omega_{*}(X)$, $\sq{\omega_{*}}(X)$, $\omega_{*,\bG}(X)$, $\sq{\omega_{*,\bG}}(X)$: Double point cobordism, double point cobordism $- \otimes_{\LL} \ZZ$, double point cobordism with principal $\bG$-bundles, same $-\otimes_{\LL} \ZZ$.
  \item If $M$ is an abelian group, $M_{\QQ} = M \otimes_{\ZZ} \QQ$ and $M\otg = M \otimes_\ZZ \ZZ\otg$.  If $\phi: M \to M'$ is a homomorphism of abelian groups, $\phi_{\QQ} = \phi \otimes_{\ZZ} \QQ$ and $\phi\otg = \phi \otimes_\ZZ \ZZ\otg$.  There is an \emph{exception} to this notation: If $M$ is complete with respect to a filtration (e.g., $\Omega^*(B\bG)$ with respect to the coniveau filtration, $MU^*(B\bG(\CC))$ with respect to the skeletal filtration, etc.) then $M_\QQ$ and $M\otg$ are to be interpreted in terms of completed tensor product.
\end{itemize}


\section{Background}\label{sec:bg}

\subsection{Torsion primes}\label{ss:torsion} We recall here some classical facts about torsion in cohomology of Lie groups.
\begin{prop} Suppose $G$ is a compact connected Lie group and $p$ is a prime.  Then, TFAE
  \begin{enumerate} \item $H^*(G, \ZZ)$ has non-trivial $p$-torsion.
    \item $H^*(BG, \ZZ)$ has non-trivial $p$-torsion.
    \item The cokernel of the restriction map $H^*(BT, \ZZ) \to H^*(G/T, \ZZ)$, associated to the fibration $G/T \to BT \to BG$, has non-trivial $p$-torsion.
    \item There is an elementary abelian $p$-subgroup of $G$ that is not contained in a maximal torus.
  \end{enumerate} 
\end{prop}

\begin{na}\label{na:char-hom} Suppose $\lambda \in \oh{\bT}$ is a character of $\bT$.  It determines a character of $\bB$ by extending trivially along the unipotent radical $\bN \subset \bB$, and an associated line bundle $\L(\lambda) = (\AA^1_\lambda \times \bG)/\bB$ on $\bG/\bB$.  Define the \emph{characteristic homomorphism} \[ \ch: \Sym \oh{\bT} \to \CH^*(\bG/\bB) \qquad \text{by} \qquad \ch(\lambda) = c_1(\L(\lambda)). \]  This is an algebro-geometric analog of the restriction map in (iii) above.  It is rationally surjective.
\end{na}

\begin{defn} The \emph{torsion index} of a connected reductive group $\bG$, denoted $\tau_{\bG}$, is the least positive integer that kills the cokernel of the characteristic homomorphism $\ch$.   If $\bG$ is not necessarily connected, set $\tau_\bG = \tau_{\bG_0} \times \Card{\bG/\bG_0}$.

  We say that a prime $p$ is a \emph{torsion prime} for a reductive group $\bG$ if $p$ divides $\tau_{\bG}$.  This is equivalent to $p$ either dividing $\Card{\bG/\bG_0}$ or satisfying one of the equivalent conditions of the Proposition (for a compact real form of $\bG$).
\end{defn}

\begin{remark} In effect, the torsion primes measure the degree to which we cannot mimic the usual Leray-Hirsch argument which shows that $H^*(BU(n),\ZZ) \hookrightarrow H^*(BU(1)^n, \ZZ)^{S_n}$ for the analogous map $H^*(B\bG) \to H^*(B\bT)^W$: It is precisely the failure of surjectivity in (iii) that prevents Chern classes from providing us with cohomology classes on the total space satisfying the hypotheses of Leray-Hirsch!
\end{remark}

\begin{remark} A linear algebraic group $\bH$ is said to be \emph{special} if every (\'etale-locally trivial) principal $\bH$-bundle is Zariski-locally trivial.   For instance, $\bH = \GG_a$, $\GG_m$, $\mathbf{SL_r}$, and $\mathbf{Sp_r}$ are special, and they are essentially (i.e., up to extensions) the only examples.  Since $\bB$ is always special, it is not too hard to verify that a reductive group $\bG$ is special iff all $\bG$-bundles $\P$ admit Zariski-local $\bB$-reductions iff for all $\P$ the generalized flag bundle $(\bG/\bB)_{\P}$ has Zariski-local sections.  Using this, it's not hard to show that a reductive connected group $\bG$ is special iff $\tau_{\bG} = 1$.
\end{remark}

The last remark can be amplified via the following geometric reformulation of the relationship between the flag manifold $\bG/\bB$ and torsion, which is essentially Th\'eor\`eme~2 of \cite{Grot}:
\begin{prop}\label{prop:BGsect} There exist integers $d_i > 0$ with $\gcd (d_i) = \tau_{\bG}$ satisfying the following property: For any  principal $\bG$-bundle $\P \to X$, there exist closed subschemes $Z_1, \ldots, Z_k \subset (\bG/\bB)_{\P} = \P/\bB$ such that $Z_i$ is generically finite of degree $d_i$ over $X$, $i=1,\ldots,k$.  In fact, there exist representations $V_1,\ldots,V_k$ of dimension $d = \dim \bG/\bB$ such that $[Z_i] = c_d\left((V_i)_{\P}\right) \cap [(\bG/\bB)_{\P}]$.
\end{prop}

\begin{remark} At least if $\bG$ is connected Grothendieck also proved the converse: There exists a principal $\bG$-bundles $\P \to X$ for which any such degree \emph{must} be divisible by $\tau_{\bG}$.
\end{remark}

\subsection{Algebraic Cobordism}\label{ss:algcob}
We summarize results on algebraic cobordism that we need:
\begin{prop}[Levine-Morel]\label{prop:lm}\mbox{}\begin{enumerate} \item $\omega_*(-)$ is an ``oriented Borel-Moore homology theory.''  That is, it has projective pushforwards, lci pullbacks, Chern classes, and these satisfy the usual compatibilities (e.g., base-change relations in Cartesian squares, push-pull relations for Chern classes of pullback bundles, formal group law for $c_1$ of a tensor product) and conditions (right-exact localizations sequence, $\AA^1$-homotopy invariance).
  \item The formal group law on $\omega_{*}(\Spec k)$ induces an isomorphism $\LL \isom \omega_{*}(\Spec k)$.
  \item Universality of $\omega_*(X)$ (as oriented Borel-Moore homology theory) induces a natural homomorphism $\omega_*(X) \to \CH_*(X)$.  This induces an \emph{isomorphism}, compatible with the structures in (i),
    \[ \sq{\omega}_*(X) = \omega_*(X) \otimes_{\LL} \ZZ \isom \CH_*(X) \]
  \item (A special case of (iii)) Suppose $f: Y \to X$ is generically finite of degree $d$.  Then, 
    \[ f_* [Y] = d [X] + \mm \omega_*(X) \qquad \text{i.e.,} \qquad  f_* [Y] = d [X] \in \sq{\omega}_*(X).\]
\end{enumerate}
\end{prop}

We recall the double point degeneration package:
\begin{na} We say that $t: \Y \to \PP^1$ is a \emph{double point degeneration} if $\Y \in \Sm_k$ is of pure dimension, $t$ is flat, and for any non-regular value $\zeta \in \PP^1$ the fiber $\Y_\zeta = t^{-1}(\zeta)$ decomposes as
  \[ \Y_\zeta = A \cup_D B \] where $A$ and $B$ are smooth divisors in $\Y$ intersecting transversely along $D$.  Note that $A, B, D$ are allowed to be disconnected or empty. (This holds trivially at regular values.)

  Let $\cN_{A/D}$ (resp., $\cN_{B/D}$) denote the normal bundle of $D$ in $A$ (resp., $B$).  One can check that $\cN_{A/D} \otimes \cN_{B/D} \isom \O_D$ and so the projective bundles \[ \PP_D(\O_D \oplus \cN_{A/D}) \to D \qquad \text{and} \qquad \PP_D(\O_D \oplus \cN_{B/D}) \to D \] are isomorphic, since the vector bundles being projectivized differ by twisting by a line bundle.  Let $\PP_D(\cN)$ denote either of these.
\end{na}

We now give the definition of $\omega_{*,\bH}$ along the lines of \cite{LP1} and \cite{LP}: 
\begin{defn}
For $X \in \Sch_k$, let \[ \Cycle_{*,\bH}(X) \eqdef \ZZ \cdot \left\{\left[\pi:Y \to X, \P\right] : \begin{gathered} \text{$Y \in \Sm_k$, irreducible of dim. $*$}\\\text{$\pi$ projective}\\\text{$\P$ principal $\bG$-bundle on $Y$}\end{gathered}\right\} \] denote the free abelian group on isomorphism classes of pairs $[\pi: Y \to X, \P]$ as indicated.  Note that not-necessarily irreducible $Y$ also have well-defined cycles, by taking a sum (in the group structure) over their connected components.
  
  Define $\omega_{*,\bH}(X)$ as the quotient graded abelian group
  \[ \omega_{*,\bH}(X) = \frac{\Cycle_{*,\bG}(X)}{\text{double point cobordism relations}} \] where by double point cobordism relations we mean the following:  Suppose $\Y \in \Sm_k$ is of pure dimension, $\P$ is a $\bG$-bundle on $\Y$, and $(\pi, t): \Y \to X \times \PP^1$ is a projective morphism such that $t: \Y \to \PP^1$ is a double double degeneration.  Suppose $0 \in \PP^1$ is a regular value, and $1 \in \PP^1$ not necessarily so.  Writing $\Y_{1} = A \cup_D B$ as above, the resulting \emph{double point cobordism relation} is
  \[ [\Y_0 \to X, \res{\P}{\Y_0}] = [A \to X, \res{\P}{A}] + [B \to X, \res{\P}{B}] - [\PP_D(\N) \to X, \res{\P}{\PP_D(\N)}]. \]
\end{defn}

\begin{na} Taking $\bH = \{\id\}$ this specializes to a definition of $\omega_{*}(X)$.  There are external product maps $\omega_{*,\bH}(X) \times \omega_{*,\bH'}(X') \to \omega_{*,\bH \times \bH'}(X \times X')$, and in particular each of these groups is a module over $\LL = \omega_{*}(k)$ by taking external products.
\end{na}

\begin{constr}\label{constr:vartheta} We now construct the map $\vartheta: \omega_{*,\bH} \to \MU_*(B\bH(\CC))$ of Theorem~\ref{thm:main}.  
  
  The construction relies on the following geometric description of the complex bordism $\MU_*$: 
  \[ \MU_*(X) = \frac{ \ZZ \cdot \left\{ [f: M \to X]: \begin{gathered} \text{$M$ is a closed stably almost-complex manifold}\\\text{and $f: M \to X$ is a continuous map}\end{gathered}\right\}}{\text{the usual cobordism relations}} \]

    We first define $\vartheta$ on the generators of $\omega_{*,\bH}$: $[Y \to \pt, \P]$ gives rise to the stably almost-complex manifold $Y(\CC)$ and a continuous map, defined up to homotopy, $f: Y(\CC) \to B\bH(\CC)$ classifying the $\bH$-bundle $\P$; the cobordism relations show that homotopic maps give the same element of $\MU_*(B\bH(\CC))$, so that we may unambiguously define
    \[ \vartheta\left([Y \to \pt, \P]\right) = [f: Y(\CC) \to B\bH(\CC)] \]
    It remains to verify that the double-point relations are satisfied so that $\vartheta$ descends to a map from $\omega_{*,\bH}$.  This follows by the usual argument that double point relations hold in complex cobordism.
\end{constr}


\section{Characteristic classes for $\bG$ bundles}\label{sec:charBG}
\subsection{Some notation}\label{ss:char-notation} We recall some convenient notation from \cite{CPZ} for dealing with algebras of Chern classes in the presence of a formal group law.

\begin{na} Suppose that $R$ is a (graded) $\LL$-algebra, that is a (graded) algebra equipped with a (graded) formal group law $F \in R\ps{x,y}$ ($-2 = \deg x = \deg y = \deg F$).  For any complete $R$-algebra $R'$, the formal group law gives rise to an abelian group 
  $(R', +_F)$ by defining $a +_F b = F(a,b) \in R'$. This determines a functor from complete $R$-algebras (with continuous homomorphisms) to abelian groups.
\end{na}

\begin{defn}  Suppose $M$ is an abelian group.  Define the \emph{twisted group algebra} $R\ps{M}$ by the universal property
  \[ \Hom_{\text{compl.}\atop{R'\text{alg}}}(R\ps{M}, R') = \Hom_{\text{gp}}(M, (R',+_F)). \]  There is also a variant suitably taking into account gradings.  We may explicitly construct $R\ps{M}$, together with a grading and filtration, as follows:
\end{defn}

\begin{constr} Let 
  \[ R\sq{\ps{M}} \eqdef R\left[x_m: m \in M\right]/\left( \begin{gathered} x_0 = 0\\ -_F x_m = x_{-m}  \\ x_m +_F x_{m'} = x_{m+m'}\end{gathered}\right) \]
    Elementary properties of formal group laws imply that it is a \emph{filtered} ring by the \emph{internal degree}: $\sq{I}_{\geq k}$ is spanned over $R$ by the monomials $x_{m_1} \cdots x_{m_{k'}}$ with $k' \geq k$.  
    Our grading conventions for $F$ imply that it is also a \emph{graded} ring via the \emph{total grading}: $\deg r x_{m_1} \cdots x_{m_k} = \deg r - 2 k$.
    
    Now let $R\ps{M}$ be the completion of $R\sq{\ps{M}}$ at the ideal $\sq{I}_{\geq 1}$.  We remark that it is graded (by \emph{total grading}), filtered (by \emph{internal degree}), and complete with respect to the filtration.
\end{constr}

\begin{example}\mbox{}
  \begin{itemize}
    \item Suppose $R_{\GG_a}$ is any graded algebra equipped with the \emph{additive} formal group law $F(x,y) = x+y$.  Then,
      \[ R_{\GG_a}\ps{M} = \oh{\Sym} M \] is the completed symmetric algebra.  In this case, the filtration by internal degree is a grading and the ring is bi-graded.
    \item Suppose $R_{\GG_m}$ is any (graded) algebra equipped with the \emph{multiplicative} formal group law $F(x,y) = x+y + \beta x y$, with $\beta \in R$ an invertible element of degree $2$.  (If we're willing to lose the grading, we can forget about $\beta$.)  Then, there is an isomorphism
      \[ R_{\GG_m}\ps{M} \isom \oh{R_{\GG_m}[M]} \qquad \text{where} \qquad R_{\GG_m}[M] = R_{\GG_m}\left[q^m : m \in M\right]/\left(q^{0}-1, q^{-m} q^m - 1, q^{m+m'}-q^m q^{m'} \right) \] with the completed group algebra.  The isomorphism is given by
      \[ x_m \mapsto  \frac{q^m-1}{\beta} \qquad \text{and} \qquad  q^m \mapsto 1 + \beta x_m   \] 
    \item Suppose $M = \ZZ^r$, with basis $\lambda_i = (\ldots,0,1,0,\ldots)$, $i=1,\ldots,r$.  Then, there is an isomorphism 
      \[ R\ps{t_1,\ldots,t_r} \isom R\ps{M}  \quad \deg t_i = -2, \qquad\text{given by $t_i \mapsto x_{\lambda_i}$}\] However, the map $M \to R\ps{M}$ ($m \mapsto x_m$) is complicated from this point of view: If $m = \sum n_i \lambda_i \in M$, then 
      \[ [n_1]_F t_1 +_{F} [n_2]_F t_2 +_F \cdots +_F [n_r]_F t_r \mapsto x_m. \]
    \item Suppose $\bT$ is a torus, so that the group $\oh{\bT}$ of characters is isomorphic as abelian group to $\ZZ^r$ for some $r$.  Choosing a basis of characters $\lambda_1, \ldots, \lambda_r$, $\oh{\bT} \isom \ZZ^r$, puts us in the previous situations.  That is, it gives rise to an isomorphism $R\ps{\oh{\bT}} \isom R\ps{t_1,\ldots,t_r}$, but the map $\oh{\bT} \to R\ps{t_1,\ldots,t_r}$ involves writing a character in terms of the fixed basis and then using the formal group law (e.g., $\lambda_1 \otimes \lambda_2$ goes  $t_1 +_F t_2 = t_1 + t_2 + (\text{higher filtration})$, $\lambda_1^{\otimes -1}$ goes to $-_F t_1 = -t_1 + (\text{higher filtration})$, etc.).
  \end{itemize}
\end{example}

Finally, we record a convenient Lemma for working with graded modules:
\begin{lemma}\label{lem:grNAK} Suppose $R$ is a non-negatively graded ring, and let $R_+$ denote the ideal of positively graded elements of $R$.  Suppose $M, N$ are non-negatively graded $R$-modules, and $\phi: M \to N$ a graded $R$-module homomorphism. \begin{enumerate}
  \item $\phi$ is surjective iff $\ol{\phi}: M/R_+ M \to N/R_+ N$ is surjective.
  \item $\phi$ is an isomorphism iff there exists a graded $R$-module isomorphism $\varphi: M \isom N$ such that $\ol{\phi}=\ol{\varphi}: M/R_+ M \to N/R_+ N$ coincide.
\end{enumerate}
\end{lemma} \begin{proof}\mbox{}\begin{enumerate} \item Follows from the graded Nakayama's Lemma (or an explicit induction on the graded degree). 
  \item The ``only if'' implication is trivial, and it suffices to prove ``if.''  Surjectivity of $\phi$ follows by applying (i).  Replacing $\phi$ by $\phi \circ \varphi^{-1}$, we may assume $\varphi = \id$.  Our hypothesis is then equivalent to the condition that $m - \phi(m) \in R_+ M$ for all $m \in M$.  To prove injectivity it suffices to show that $\ker \phi \subset (R_+)^N M$ for all $N$ for then $\ker \phi = 0$ by grading reasons.

  We prove this by induction on $N$.  Suppose $m \in (R_+)^N M \cap \ker \phi$.  Then, we may write $m = \sum r_i m_i$, with $r_i \in (R_+)^N$.  By assumption, $m_i - \phi(m_i) \in R_+ M$ for all $i$ so that
  \[ m = m - \phi(m) = \sum r_i (m_i - \phi(m_i)) \in (R_+)^{N+1} M. \qedhere\]
  \end{enumerate}
\end{proof}

\subsection{Overview}
\begin{na} For $\bG$ connected reductive, the Schubert (algebraic cell) decomposition determines the additive structure of $\omega_*(\bG/\bB)$ as $\LL$-module: It is free of rank $\Card{W}$, with $w \in W$ contributing a generator in dimension $\ell(w)$.  However, in order to benefit from the \emph{orientability} of $\omega_*$ we will need to exploit the relationship between this additive structure and Chern classes of line bundles.  It turns out we can get away with just statements about the torsion index, but this is just the tip of a richer picture.\footnote{Due to Bernstein-Gelfand-Gelfand for homology, Demazure for K-theory (and homology?), Bressler-Evens for complex cobordism.  This picture is developed in algebraic cobordism in \cite{CPZ}.}
\end{na}

\begin{defn} Suppose $\P^\bT \to Y$ is a principal $\bT$-bundle.  A character $\lambda \in \oh{\bT}$ determines a $1$-dimensional representation $\AA^1_\lambda$ of $\bT$ and hence an \emph{associated line bundle} $\L(\lambda) = (\AA^1_\lambda)_{\P^\bT}$ on $Y$.  There is a unique $\LL$-algebra homomorphism
  \[ c_{\bT}(-, \P^\bT): \LL\ps{\oh{\bT}} \to \End_\LL\left(\omega_*(Y)\right) \qquad \text{determined by} \qquad c_{\bT}(\lambda, \P^\bT) \cap - = c_1(\L(\lambda)) \cap -. \]
(By the usual abuse of notation, we denote application of the endomorphism $c_{\bT}(A, \P^\bT)$ by $c_{\bT}(A, \P^\bT) \cap -$.)
\end{defn}

\begin{na}\label{na:char-hom2} Suppose $\P \to Y$ is a principal $\bG$-bundle, $p: (\bG/\bB)_{\P} \to Y$ the associated $\bG/\bB$-bundle, and $\P^{\bT}$ the $\bT$-bundle on $(\bG/\bB)_{\P}$.  Consequently, there is a homomorphism $c_{\bT}(-,\P^\bT): \LL\ps{\oh{\bT}} \to \End\left(\omega_*((\bG/\bB)_\P)\right)$.  This is the cobordism, in-family, version of the \emph{characteristic homomorphism} of \ref{na:char-hom}.
\end{na}

\begin{na}
Our goal will be to refine (first rationally, then over $\ZZ\otg$) $c_{\bT}$ to \[ c_{\bG}: \LL\otg\ps{\oh{\bT}}^W \longrightarrow \End_{\LL\otg}(\omega_*(Y)\otg).\] 
\end{na}

\begin{defn}\label{defn:char-class} An assignment, for all $X \in \Sch_k$,
  \[ c(-): \left\{ \text{Principal $\bG$-bundles on $X$} \right\} \to \End(\omega_*(X)) \] is a characteristic class if it satisfies
  \begin{enumerate}
    \item (Compat. with pushforward.) Suppose $\pi: Y \to X$ is projective, and $\P$ a $\bG$-bundle on $X$.  Then, \[ \pi_*\left(c(\pi^* \P) \cap x\right) = c(\P) \cap \pi_* x. \]
    \item (Compat. with pullback.) Suppose $f: Y' \to Y$ is a map with $Y, Y' \in \Sm_k$, and $\P$ a $\bG$-bundle on $Y$.  Then,
      \[ f^*\left( c(\P) \cap x \right) = c(f^* \P) \cap f^* x. \]
      \end{enumerate}
\end{defn}

\begin{na}\label{na:char-prod}
 The two conditions above imply the following compatability with external product: Suppose $Y,Y' \in \Sm_k$, $\P$ a $\bG$-bundle on $Y$, and $p_1: Y \times Y' \to Y$ the projection.  Then,
      \[ c(p_1^* \P) \cap (y \times y') = (c(\P) \cap y) \times y'. \]  (Indeed, (ii) implies this for $y' = [Y']$.  Then, (i) extends this to pushforwards of such classes, which generate.)

 For $Y\in \Sm_k$ is smooth, it follows that $c(\P) \in \End(\omega_*(Y))$ is  given by capping with $c(\P) \cap [Y]$ regarded as an element of $\Omega^*(Y)$:
  \[ c(\P) \cap y = \Delta^* p_1^*(c(\P) \cap y) = \Delta^*\left(c(p_1^* \P) \cap ([Y] \times y)\right) = \Delta^*\left( (c(\P) \cap [Y]) \times y \right) \]
\end{na}

\begin{prop}\label{prop:BG} Characteristic classes for $\bG$-bundles are in bijection with $\Omega^*(B\bG)$, via $c \mapsto c(E\bG) \cap [B\bG]$.  Similarly, characteristic classes with coefficients in $R = \ZZ\otg$ or $\QQ$ are in bijection with $\Omega^*(B\bG) \oh{\otimes} R$.
\end{prop}
\begin{proof} By (i), every characteristic class is determined by its restriction to $\Sm_k$ since $\omega_*(X)$ is generated by pushforwards from smooth schemes.  Injectivity would follow from (ii) + \ref{na:char-prod}, if every $\bG$-bundle $\P$ on a smooth scheme were pulled back from (a finite approximation to) the universal bundle on $B\bG$. This is not quite true, but there is a well-known workaround: By Jouanolou's trick, we may replace $Y \in \Sm_k$ (recall, this includes quasi-projective!) by a smooth \emph{affine} scheme $Y'$ which is $\AA^1$-equivalent to it (in fact, a torsor for a vector bundle over $Y$); then $\res{\P}{Y'}$ \emph{will} be pulled back from $B\bG$ (or rather, a finite approximation) by a map which is unique up to $\AA^1$-homotopy (in a potentially larger approximation).  So the assignment is injective.  

Reversing the previous argument shows that it is also surjective: To define $c$ on smooth schemes, we need only verify independence of the choice of affine scheme $Y'$ and map to $B\bG$.  By taking fiber products to dominate any two choices of $Y'$, we reduce to showing independence of the map to $B\bG$, which follows by the claim above that the map is unique up to $\AA^1$-homotopy.  Verifying that property (i) holds on smooth schemes (by general push-pull properties), we extend to arbitrary varieties by (i).
\end{proof}

\subsection{Rationally}
\begin{na} Applying $c_\bT$ to the ``universal bundle'' $E\bT \to B\bT$ induces an isomorphism $\LL\ps{\oh{\bT}} \isom \Omega^*(B\bT)$, and analogously with coefficients in $\ZZ\otg$ or $\QQ$.  It is no surprise that $\Omega^*(B\bG)_\QQ \isom \Omega^*(B\bT)_\QQ^W \isom \LL_\QQ\ps{\oh{\bT}}^W$.  (Between when we started writing this document and the present, a reference for this fact has appeared in \cite{Krishna}.)  Combining these facts with Prop.~\ref{prop:BG}, we obtain the following Lemma.
\end{na}

\begin{lemma}\label{lem:cBG-Q} Suppose $X \in \Sch_k$, and $\P$ is a principal $\bG$-bundle on $X$. There is an $\LL_\QQ$-module homomorphism
  \[ c_{\bG}(-,\P): \LL_\QQ\ps{\oh{\bT}}^W \to \End(\omega_*(X)) \] satisfying the conditions of Defn.~\ref{defn:char-class}.
  
  Suppose furthermore that $\P$ has a $\bB$-reduction $\P^\bB$, and let $\P^{\bT} = (\bT)_{\P^{\bB}}$ be the associated principal $\bT$-bundle.  Then, $c_{\bG}(-,\P) = c_{\bT}(-, \P^{\bB})$.
\end{lemma}
\begin{proof} A $\bB$-reduction induces a factorization of the map $Y' \to B\bG$ (in the proof of Prop.~\ref{prop:BG}) through $B\bT$, and there is a well-known isomorphism $\Omega^*(B\bT) \isom \Omega^*(B\bB)$ (since $B\bT \to B\bB$ is a Zariski-locally trivial bundle with fibers $\bB/\bT \isom \bN$ scheme-theoretically isomorphic to affine spaces).
\end{proof}

\begin{na}\label{na:cG-split} In light of the last sentence of the Lemma, we may use the splitting principal (see Lemma~\ref{lem:equaliz} for one formulation) to give a description of $c_{\bG}$.  Suppose $X$ and $\P$ are as in the Lemma, and consider $p: (\bG/\bB)_{\P} \to Y$; the $\bG$-bundle $p^* \P$ has a $\bB$-reduction $\P^\bB$, with associated $\bT$-bundle $\P^\bT$.  Then,
    \[ c_{\bT}(A,\P^\bT) \cap p^* x = p^*(c_{\bG}(A, \P) \cap x) \] and this requirement uniquely determines $c_{\bG}$ since $p^*$ is injective.  Then, the claim of the Lemma is essentially equivalent to the claim that for $A \in \LL_\QQ\oh{\ps{\bT}}^W$, the operation $c_{\bT}(A, \P^\bT)$ preserves $\im p^* \subset \omega_*((\bG/\bB)_{\P})_\QQ$.  The above proof ``reduces this to the universal example,'' but it is possible to avoid this detour and give a more direct argument by constructing $c_{\bG}$ directly using Chern classes of vector bundles associated to representations (analogous to Vistoli's original construction in Chow theory).
\end{na}
  
\subsection{Better than rationally}\label{ss:charBG-2}
\begin{na} By Remark~\ref{rmk:depends}, and its footnote, this section is not strictly necessary for the proofs of our results.  However, we include it as it may be of independent interest.
\end{na}

We have the following amusing strengthening of the ``splitting principle'' for general flag bundles:
\begin{lemma}\label{lem:equaliz} The diagram
  \[ \xymatrix{\omega_*(Y)\ar[r]^-{p^*} & \omega_{*+d}\left( (\bG/\bB)_{\P}\right) \ar@<1ex>[r]^-{p_1^*} \ar@<-1ex>[r]_-{p_2^*} & \omega_{*+2d}\left( (\bG/\bB)_\P \times_Y (\bG/\bB)_\P\right) } \] becomes a split equalizer diagram after inverting $\tau_\bG$. (Here $d = \dim \bG/\bB$.)
\end{lemma}
\begin{proof}
  Let $A \in \LL\ps{\oh{\bT}}$, in filtration $d$ and degree $d$, be such that $p_*(c_\bT(A) \cap [\bG/\bB]) = \tau_{\bG} \in \CH_0(\pt)$.  Define
  \[ s: E_*( (\bG/\bB)_{\P}) \to E_{*-d}(Y) \qquad \text{by} \qquad s(x) = p_*(c_\bT(A, \P^\bT) \cap x) \]
  \[ t: E_*( (\bG/\bB)_\P \times_Y (\bG/\bB)_\P) \to E_{*-d}( (\bG/\bB)_\P) \qquad \text{by} \qquad t(x) = (p_1)_*\left( c_\bT(A, (p_2)^* \P^\bT) \cap x\right). \]
  I claim that $s \circ p^* \equiv \id \pmod{\mm}$, $t \circ (p_1)^* \equiv \tau_{\bG} \id \pmod{\mm}$, and $t \circ (p_2)^* = p^* \circ s$ (on the nose); in other words, up to $\mm = \LL_{\geq 1}$, the diagram is a \emph{split} equalizer.  Lemma~\ref{lem:grNAK} implies that $s \circ p^*$ and $t \circ (p_1)^*$ are isomorphisms and (modifying $s$ and $t$ by these isomorphisms) the diagram is a split equalizer.
  
More explicitly:  The composites agree, so it suffices to prove that $p^*$ is injective and that $p^*$ is surjective onto the equalizer $K = \ker(p_1^* - p_2^*)$.  By Lemma~\ref{lem:grNAK}, $s \circ p^*$ is an isomorphism and in particular $p^*$ is injective.  Note that $p^* \circ s$ restricts to an endomorphism of $K$, and by the above it coincides with the identity modulo $\LL_{\geq 1}$; by Lemma~\ref{lem:grNAK}, it is an isomorphism and in particular surjective.

Now we prove the claim.  The compatibility of Chern classes with pullback and ``base-change'' for pullbacks-pushforwards implies
\[ t \circ p_2^*(x) = (p_1)_*\left(c_\bT(A, p_2^* \P^\bT)\cap  p_2^* x\right) = (p_1)_* p_2^*\left(c_\bT(A, \P^{\bT}) \cap x\right) = p^* p_*\left(c_{\bT}(A, \P^{\bT}) \cap x\right) = p^* \circ s(x). \]
By Prop.~\ref{prop:lm}, it suffices to prove the other two equalities in Chow theory.  These are standard, but we include a proof for completeness:

We may as well assume $x = f_* [X]$ for $f: X \to Y$ a map from a smooth connected variety and $[X] \in \CH^{\dim X}(X)$ the fundamental class, since such elements span Chow.  Let $p'$ and $f'$ be defined by the pullback square
\[ \xymatrix{ 
(\bG/\bB)_{f^* \P}\ar[d]_{p'} \ar[r]^{f'} & (\bG/\bB)_{\P}\ar[d]^p \\
X \ar[r]_f & Y } \] We first reduce to proving the analogous claim over $X$, and only for $[X]$: Since
\begin{align*} s \circ p^*(x) &= p_*\left(c_\bT(A,\P^\bT) \cap p^* f_*[X]\right) =p_*\left(c_\bT(A,\P^\bT) \cap f'_* (p')^* [X]\right)\\
&=p_* f'_*\left(c_\bT(A,(f')^* \P^\bT) \cap (p')^* [X]\right)  =f_* p'_* \left(c_\bT(A,(f')^* \P^\bT) \cap (p')^* [X]\right)  \end{align*}
it suffices to show that $p'_* \left(c_\bT(A,(f')^* \P^\bT) \cap (p')^* [X]\right) = \tau_{\bG} [X]$. We're asking for an equality of elements in $\CH^{\dim X}(X) \isom \ZZ$, i.e., numbers. We may check this equality after base-change to an \'etale open trivializing $\P$ (or a closed point $x \in X$), where it follows by hypothesis on $A$.  The third computation is analogous.
\end{proof} 

\begin{corollary}\label{corr:charBG2} \begin{enumerate}
  \item $\Omega^*(B\bG)\otg \isom \Omega^*(B\bT)\otg^W$.
  \item Suppose $X \in \Sch_k$, and $\P$ is a principal $\bG$-bundle on $X$. There is an $\LL\otg$-module homomorphism
  \[ c_{\bG}(-,\P): \LL\otg\ps{\oh{\bT}}^W \to \End(\omega_*(X)\otg) \] satisfying the conditions of Defn.~\ref{defn:char-class}.
  
  Suppose furthermore that $\P$ has a $\bB$-reduction $\P^\bB$, and let $\P^{\bT} = (\bT)_{\P^{\bB}}$ be the associated principal $\bT$-bundle.  Then, $c_{\bG}(-,\P) = c_{\bT}(-, \P^{\bB})$.
\end{enumerate}
\end{corollary}
\begin{proof} The proof of (ii) is exactly the same as Lemma~\ref{lem:cBG-Q}.  
  
Before giving a rigorous proof of (i), we sketch the idea: We will apply Lemma~\ref{lem:equaliz} with $Y = B\bG$, $\P = E\bG$, $(\bG/\bB)_{\P} = B\bT$.  This tells us that the pullback map $p^*$ is a split injection, and its image is certainly contained in the Weyl invariants.  Letting $s'$ be such that $s' \circ p = \id$, we have $\im p^* = \ker(p \circ s' - \id)$.  We saw above that rationally the Weyl invariants are in this kernel; since $\Omega^*(B\bT)$ is torsion-free, this proves the reverse containment. Making this rigorous will require dealing with filtration and completion issues.

Recall that we regard $\Omega^*(B\bT)\otg$ and $\Omega^*(B\bG)\otg$ as filtered by their coniveau filtrations.  The maps $p^*$ and $s = p_*(c_\bT(A, \P) \cap -)$ both preserve coniveau filtration,\footnote{In general, pullback preserves coniveau filtration, Chern classes of line bundles increase it by $1$, and pushforward decreases it by relative dimension. So, $c_\bT(A,\P)$ increases it by $d$ and $p_*$ decreases it by $d$.} and so actually determine a split inclusion
\[\xymatrix@R=2pc{ \dfrac{\Omega^* B\bG}{F^j \Omega^* B\bG}\otg  \ar@<1ex>[r]^{p^*} & \dfrac{\Omega^* B\bT}{F^j \Omega^* B\bT}\otg  \ar@{..>}@/^3ex/[l]^{s} }\]  for each $j$.  (This is just a way of rewording the approximation argument in a way independent of the chosen approximations.)  

For notational convenience, set $M = \LL\otg\ps{\oh{\bT}}$, $M_\QQ = \LL_\QQ\ps{\oh{\bT}}$, $N = \Omega^*(B\bG)\otg$, $N_\QQ=\Omega^*(B\bG)_\QQ$.  We saw that $s \circ p^*: N \to N$ was filtration preserving and invertible, so that $s' = (s \circ p^*)^{-1} \circ s: M \to M$ is a filtration preserving map satisfying $s' \circ p = \id$.  Then, $h = \id - p^* s': \Omega^*(B\bT)\otg \to \Omega^*(B\bT)\otg$ preserves the filtration, so that $\im p^* = \ker h$ on each quotient $M/F^j M$ (and of course on $M$).  It now suffices to show that $h$ vanishes on $M^W$, and since $M$ injects into $M_\QQ$ it suffices to show that $h_\QQ$ vanishes on the image of $M^W$ in $M_\QQ$; but this image is contained in $(M_\QQ)^W = N_\QQ$, so we're done.
\end{proof}

\begin{na} The splitting principle description of $c_\bG$ of \ref{na:cG-split} is of course still valid.  In this case, however, a direct approach (avoiding the reduction to the universal case) is more difficult as it was very convenient to make use of torsion-freeness which certainly need not hold in $\omega_*(X)$ for arbitrary $X$.  
    
There should however be the following more direct approach, giving in fact this type of results in slightly more generality, e.g., only no $2$-torsion. (We don't need this approach, so we haven't carried this out in detail!): A family version of the analysis of \cite{CPZ} should allow one to equip $\omega^*( (\bG/\bB)_\P)\otg$ with an action of the BGG/Demazure-style operators compatibly with those on $\LL\otg\ps{\oh{\bT}}$, and to identify the rightmost maps of the equalizer diagram of Lemma~\ref{lem:equaliz} with the ``action'' maps for this.\footnote{The action of BGG/Demazure-style operators can be viewed as a ``deformation'' of the Weyl group action.  In the case of cobordism (anything but cohomology and $K$-theory), Bressler \& Evens showed that they even cease to satisfy the familiar \emph{braid relations}, so that $D_w$ depends on writing $w$ as a product of simple roots and not merely on $w$.} This identifies $\omega^*(Y)\otg \subset \omega^*( (\bG/\bB)_\P)\otg$ with the invariants for this action, and reduces us to checking that the invariants for this action in $\LL\otg\ps{\oh{\bT}}$ are $\LL\otg\ps{\oh{\bT}}^W$ using torsion-freeness.  A similar argument is possible by directly proving a ``flag bundle theorem.''
    
The following Proposition encodes the goal of such an approach. (It's well-known and less interesting rationally, and is probably well-known to some but not all in the form below.  We don't use the proposition, so include only a brief sketch of proof already assuming the existence of $c_\bG$.)
\end{na}

\begin{prop}\label{prop:flagbundle} Suppose $\P \to Y$ is a principal $\bG$-bundle, $Y \in \Sm_k$. (We do this primarily so that gradings look nice.)
  \begin{enumerate}
  \item There is a natural isomorphism \[ \LL\otg\ps{\oh{\bT}} \otimes_{\LL\otg\ps{\oh{\bT}}^W} \omega^*(Y)\otg \isom \omega^*( (\bG/\bB)_\P)\otg  \qquad \text{given by} \qquad A \otimes x \mapsto c_{\bT}(A, \P^\bT) \cap p^* x. \]
  \item There is an action of ``BGG/Demazure-type'' operators on $\omega^*((\bG/\bB)_\P)\otg$ and the invariants are precisely the image of $p^*$.
  \item The BGG/Demazure-invariants in $\LL\otg\ps{\oh{\bT}}$ are precisely $\LL\otg\ps{\oh{\bT}}^W$.
\end{enumerate}
\end{prop}
\begin{proof} \begin{enumerate} \item Assuming Corr.~\ref{corr:charBG2}, the indicated formula is well-defined.  Surjectivity is easy to prove using localization and the results of \cite{CPZ} in the case of a trivial bundle.  In light of Lemma~\ref{lem:equaliz}, we may prove injectivity after pulling back to $(\bG/\bB)_\P$ and in particular may assume $\P$ is Zariski-locally trivial.  Injectivity is not hard to prove in the case of a trivial bundle (the results of \cite{CPZ} imply that one can ``coefficient find'' using only Chern classes), and then the case where $\P$ is Zariski-locally trivial follows by an ascending induction on closed subsets of $Y$ (with the bundle trivial on the locally closed differences) as in the proof of the projective bundle Theorem in \cite{LM}.
  \item (Defining the action is entirely geometric and no harder in families; see \cite{BrEv1} or \cite{CPZ}.) Let $R = \LL\otg\ps{\oh{\bT}}$.  Using (i), we may identify the equalizer diagram (up to regrading) of Lemma~\ref{lem:equaliz} with
    \[  \xymatrix{\omega^*(Y)\otg\ar[r]^-{p^*} &  R \otimes_{R^W} \omega^*(Y)\otg \ar@<1ex>[r]^-{p_1^*} \ar@<-1ex>[r]_-{p_2^*} &  \left(R \otimes_{R^W}  \otimes R\right) \otimes_{R^W} \omega^*(Y)\otg  } \]
    where $p^*$, $p_1^*$, and $p_2^*$ are acting only on the first factors in the tensor products.  Taking the ``Demazure'' basis $A_1, \ldots, A_d$ for $ R$ over $R^W$, and letting $M$ be the middle term, we may regard $p_1^*$ and $p_2^*$ as maps $M \mapsto M^{\oplus d}$, the first one diagonally and the second one by the actions of $A_1, \ldots, A_d$ via $c_\bT(A_i, \P^\bT)$.
  \item This is essentially a result of Demazure, and holds more generally assuming only that $2$ is not a zero divisor in the coefficient ring.\qedhere
\end{enumerate}
\end{proof}


\section{Operations}\label{sec:operations} 

\begin{na}In the course of the proof, we will want to have available the following ``operations'' on double-point cobordism with principal bundles.

\begin{enumerate}\label{enum:operations}
  \item\label{it:op-ind} Suppose $\phi: \bH \to \bH'$ is a homomorphism of algebraic groups.  Define $\LL$-module homomorphisms, natural in $X$,
    \[ \ind_{\phi} = \ind_{\bH}^{\bH'}: \omega_{*,\bH}(X) \to \omega_{*,\bH'}(X) \qquad \text{by} \qquad \ind_{\phi}\left[\pi: Y \to X, \P\right] = \left[\pi: Y \to X, (\bH')_{\P}\right] \] (Recall that $(\bH')_\P$ is the principal $\bH'$-bundle associated to $\P$, $(\bH')_{\P} = (\P \times \bH')/\bH$.)  As a special case, define $\varepsilon = \ind_{\bH}^{\{\id\}}: \omega_{*,\bH}(X) \to \omega_{*}(X)$.
  \item\label{it:op-bullet} Suppose $\P$ is a principal $\bH$-bundle on $X$.  Define the $\LL$-module homomorphism
    \[ - \bullet \P: \omega_*(X) \to \omega_{*,\bH}(X) \qquad\text{by}\qquad \left[\pi: Y \to X\right] \bullet \P = \left[\pi: Y \to X, \pi^* \P\right]. \]
  \item\label{it:op-GB} Suppose $\bB \subset \bG$ is the inclusion of a Borel, $d = \dim \bG/\bB$.  Define $\LL$-module homomorphisms, natural in $X$,
    \[ (\bG/\bB)_{-}: \omega_{*,\bG}(X) \to \omega_{*+d,\bB}(X) \qquad \text{by} \qquad (\bG/\bB)_{\left[\pi: Y \to X, \P\right]} = \left[\pi \circ p: (\bG/\bB)_{\P} \to X, \P^\bB\right] \] (Recall $p: (\bG/\bB)_{\P} = \P/\bB \to Y$ is the generalized flag bundle of $\P$, and $\P^\bB$ is $\P$ viewed as a $\bB$ bundle over $\P/\bB$.)
  \item\label{it:op-cT} Suppose $A \in \LL\ps{\oh{\bT}}$, $|A|$ its degree.  Define $\LL$-module homomorphisms, natural in $X$,
    \[ \c_\bT(A) \cap -: \omega_{*,\bT}(X) \to \omega_{*-|A|,\bT}(X) \quad \text{by} \quad \c_\bT(A) \cap \left[\pi: Y \to X, \P\right] = \pi_*\left[ (c_\bT(A, \P) \cap [Y]) \bullet \P\right] \] Recall (from $\S$~\ref{sec:charBG}) that $c_{\bT}(-,\P): \LL\ps{\oh{\bT}} \to \End(\omega_*(Y))$ denotes the unique graded $\LL$-algebra homomorphism that takes a character $\lambda \in \oh{\bT}$ to $c_1(\L(\lambda))$.
    Fixing a surjection $\bB \twoheadrightarrow \bT$, we can completely analogously define \[ \c_{\bT}(A) \cap -: \omega_{*,\bB}(X) \to \omega_{*-|A|, \bB}(X) \quad \text{by} \quad \c_{\bT}(A) \cap \left[\pi: Y \to X, \P\right] = p_*\left[ (c_{\bT}(A, \P^\bT) \cap [Y]) \bullet \P\right] \] where $\P^\bT = (\bT)_{\P}$ is the $\bT$-bundle associated (via the chosen surjection) to $\P$.
  \item\label{it:op-cG} Suppose $A \in \LL\otg\ps{\oh{\bT}}^W$, $|A|$ its degree  Define $\LL$-module homomorphisms, natural in $X$,
    \[ \c_{\bG}(A) \cap -: \omega_{*,\bG}(X)\otg \to \omega_{*-|A|,\bG}(X)\otg \quad \text{by} \quad \c_{\bG}(A) \cap \left[\pi: Y \to X, \P\right] = p_*\left[ (c_{\bG}(A, \P) \cap [Y]) \bullet \P\right] \] Recall (from $\S$~\ref{sec:charBG}) that $c_{\bG}(-,\P): \LL\otg\ps{\oh{\bT}}^W \to \End(\omega_*(Y)\otg)$ are certain ``characteristic classes for $\bG$-bundles.''
  \item Combining $\c_{\bT}(A)$ (resp., $\c_{\bG}(A)$) with $\varepsilon$, define
    \[ \ip{A}{-}_{\bT}=\varepsilon\left(\c_{\bT}(A) \cap -\right):   \omega_{*,\bT}(X)\otg \to \omega_{*-|A|}(X)\otg \qquad \text{for $A \in \LL\oh{\ps{\bT}}$} \]
    \[ \ip{A}{-}_{\bG}=\varepsilon\left(\c_{\bG}(A) \cap -\right):   \omega_{*,\bG}(X)\otg \to \omega_{*-|A|}(X)\otg \qquad \text{for $A \in \LL\otg\oh{\ps{\bT}}^W$} \]
\end{enumerate}

Well-definedness of (i) and (ii) is essentially obvious, as is (vi) from (iv) and (v). The rest of this section will be devoted to verifying the well-definedness of (iii)-(v).
\end{na}

\subsection{Double-point cobordism with extra structure}\label{ssec:dpextra}
\begin{na}To keep things honest, we'll abstract out the properties needed for well-definedness.  For notational convenience, it will be useful to have a common name for variants of double-point cobordism with ``extra structure.''  So, we introduce the following bit of redundant notation.
\end{na}

\begin{defn} 
  Suppose $\F: \Sm_k^{\text{op}} \to \mathop{Sets}$ is a presheaf on the category of smooth varieties.  For $f: X' \to X \in \Sm_k^{\text{op}}$ and $\gamma \in \F(X)$, we will follow tradition and denote $\F(f)(\gamma) \in \F(X)$ by $f^* \gamma$, or if the map $f$ is clear from context simply by $\res{\gamma}{X'}$.

  For $X \in \Sch_k$, define the graded abelian groups
  \[ \Cycle_{*,\F}(X) \eqdef \ZZ \cdot \left\{\left[\pi: Y \to X, \gamma \in \F(Y)\right] : \begin{gathered} \text{$Y \in \Sm_k$, pure of dimension *}\\\text{$\pi$ projective}\\\text{$\gamma \in \F(Y)$}\end{gathered} \right\} \] \[ \omega_{*,\F}(X) \eqdef \frac{\Cycle_{*,\F}(X)}{\text{double point cobordism relations}}. \]
\end{defn}

\begin{na} Taking $\F$ to be isomorphism classes of $\bG$-bundles, we recover $\omega_{*,\bG}(X)$.  The previous construction has the same functorialities as we expect of $\omega_{*,\bG}(X)$: It is covariantly functorial for projective morphisms $f: X' \to X$, and covariantly functorial for natural transformations $\F' \to \F$.
\end{na}

\subsection{Characteristic operations} 
The construction of Chern operations $\c_\bT, \c_\bG$ will be examples of the situation handled by the following Lemma:  Elements of $\F(Y)$ give rise to operations, that are already known to be be well-defined on cobordism, $\omega_*(Y) \to \omega_{*}(Y)$, that we wish to promote to transformations $\omega_{*,\F}(-) \to \omega_{*,\F}(-)$.  
\begin{lemma}\label{lem:op-ch}
  \begin{enumerate}
    \item Suppose given operations
    \[ a_Y: \F(Y) \times \omega_*(Y) \to \omega_{*,\F'}(Y) \] for all $Y \in \Sm_k$, satisfying
    \begin{enumerate}
      \item (``Homomorphism'') $a_Y(\gamma,-): \omega_*(Y) \to \omega_{*,\F'}(Y)$ is a group homomorphism for all $\gamma \in \F(Y)$.
      \item (``Push-pull'') For all smooth $Y, Y'$, $\gamma \in \F(Y)$, and projective $f: Y' \to Y$ 
	\[ f_* \left(a_{Y'}\left(f^* \gamma,[Y']\right)\right) = a_Y\left(\gamma, f_* [Y']\right). \]
    \end{enumerate}
	
	Then, the assignment
	\[ \left(\pi: Y \to X, \gamma\right) \longmapsto \pi_*\left(a_Y\left( \gamma, [Y]\right) \right) \] 
	determines a well-defined homomorphism of abelian groups $a: \omega_{*,\F}(X) \to \omega_{*,\F'}(X)$.  Moreover, $a$ commutes with pushforwards and so determines a natural transformation of functors.
    \item 
Suppose $a_Y: \F(Y) \times \omega_*(Y) \to \omega_*(Y)$ satisfies the hypotheses of (i).  Then so does $a'_Y: \F(Y) \times \omega_*(Y) \to \omega_{*,\F}(Y)$ given by
      \[ a'_Y(\gamma, x) = a_Y(\gamma, x) \bullet \gamma. \]
  \end{enumerate}
\end{lemma}
\begin{proof} \begin{enumerate} \item Suppose $(\pi, t): \Y \to X \times \PP^1$, $\gamma \in \F(\Y)$, is a double-point relation in $\omega_{*,\F}(X)$.  Let $\Y_0$ be a smooth fiber of $t$, and $\Y_1 = A \cup_D B$ a ``double-point'' fiber of $T$.  Let $i_0, i_A, i_B, i_{\PP_D}$ the natural maps from $\Y_0, A, B, \PP_D(\N)$ to $\Y$; let $\pi_0, \pi_A, \pi_B, \pi_{\PP_D}$ be the natural maps to $X$.  We must show that 
  \begin{align*}
    (\pi_0)_* a_{\Y_0}\left(\res{\gamma}{\Y_0}, [\Y_0]\right)
  &- (\pi_A)_* a_{A}\left(\res{\gamma}{A}, [A]\right)
  - (\pi_B)_* a_{B}\left(\res{\gamma}{B}, [B]\right) \\
  &+ (\pi_{\PP_D})_* a_{\PP_D(\N)}\left(\res{\gamma}{\PP_D}, [\PP_D(\N)]\right) = 0 \in \omega_{*,\F'}(X)
\end{align*} 
But for each $T = 0,A,B,\PP_D$,
\[ {\pi_{T}}_* a_{T}\left(i_T^* \gamma, [T]\right)
= \pi_* {i_T}_* a_{T}\left(i_T^* \gamma, [T]\right) 
= \pi_* a_{\Y}\left(\gamma, {i_T}_*[T]\right) \]
by functoriality of pushforward and the push-pull relation.  Since $\pi_*$ and $a_{\Y}(\gamma, -)$ are both homomorphisms, we may rewrite the expression we are to prove is zero as
\[ \pi_* a_{\Y}(\gamma, \underbrace{[\Y_0 \to \Y] - [A \to \Y] - [B \to \Y] + [\PP_D(\N) \to \Y]}_{=0 \in \omega_{*}(\Y)}).\]
\item It suffices to check that if $Y',Y \in \Sm_k$, $\gamma \in \F(Y)$, $x \in \omega_{*,\F}(Y')$ and $f: Y' \to Y$ a projective morphism, then
     \[ f_*\left( x \bullet f^* \gamma \right) = f_*(x) \bullet \gamma \in \omega_{*,\F}(X).\qedhere\]
  \end{enumerate}
\end{proof}

\begin{corollary}\label{cor:opsBG} The operations (iv) and (v) (i.e., $\c_\bT$ and $\c_\bG$) are well-defined.
\end{corollary}
\begin{proof} Part (i) of Defn.~\ref{defn:char-class} (Corr.~\ref{corr:charBG2}) implies that Chern operations (without putting back in the bundle) satisfy the conditions part (i) of the Lemma.  Part (ii) of the Lemma then implies that the operations $\c_\bT$ and $\c_\bG$ do too.  Furthermore, $\LL$-linearity follow from compatability with exterior products (\ref{na:char-prod}).
\end{proof}

\subsection{Flag bundle}
The construction of $(\bG/\bB)_{-}$ will be an example of the situation handled by the following Lemma: The construction can be lifted to the level of cobordism cycles in a way that is compatible with base change and preserves double-point cobordisms, thereby ``encoding'' its own compatibility with double-point relations.
\begin{lemma}\label{lem:op-flag}
  Suppose given an assignment
	\[ a_Y: \F(Y) \to \Cycle_{*+d,\F'}(Y) \] 
	\[ \gamma \longmapsto a_Y(\gamma) = \left(\pi_Y(\gamma): W_Y(\gamma) \to Y, \delta_Y(\gamma)\right) \qquad \text{where $W_Y(\gamma) \in \Sm_k$, $\delta_Y(\gamma) \in \F'(W_Y(\gamma))$}\]
	for all $Y \in \Sm_k$, satisfying
	\begin{enumerate} 
      \item (``Submersion'') $\pi_Y(\gamma): W_Y(\gamma) \to Y$ is smooth, surjective, of relative dimension $d$ for all $Y \in \Sm_k$.
	  \item (``Base-change'')
	    For all $Y',Y \in \Sm_k$, $f: Y' \to Y$, $\gamma \in \F(Y)$, and $\gamma' = f^* \gamma$, there is a morphism $r_{f}: W_{Y'}(\gamma') \to W_Y(\gamma)$ so that $\sq{f}^* \delta_Y(\gamma) = \delta_{Y'}(\gamma')$ and so that the diagram
	\[ \xymatrix{
	W_{Y'}(\gamma') \ar[d]_{\pi_{Y'}(\gamma')} \ar[r]^{\sq{f}} & W_Y(\gamma) \ar[d]^{\pi_Y(\gamma)} \\
	Y' \ar[r]_f & Y } \]  is Cartesian.
    \end{enumerate}
Then, the assignment
	\[ \left(\pi: Y \to X, \gamma\right) \longmapsto \pi_*\left(a_Y\left( \gamma\right) \right) \] 
	determines a well-defined homomorphism of abelian groups $a: \omega_{*,\F}(X) \to \omega_{*+d,\F'}(X)$.  Moreover, $a$ commutes with pushforwards and so determines a natural transformation of functors.
\end{lemma}
\begin{proof} 
  With notation as in the proof of the previous Lemma, it suffices to show that
\begin{align*}
    (i_0)_* a_{\Y_0}\left(\res{\gamma}{\Y_0}\right)
  &- (i_A)_* a_{A}\left(\res{\gamma}{A}\right)
  - (i_B)_* a_{B}\left(\res{\gamma}{B}\right) \\
  &+ (i_{\PP_D})_* a_{\PP_D(\N)}\left(\res{\gamma}{\PP_D}\right) = 0 \in \omega_{*,\F'}(\Y)
\end{align*}  

The property of being a double-point degeneration can (like smoothness) be checked on a smooth surjective cover, like $\pi'=\pi_\Y(\gamma)$.  Consequently, $\Y' = W_{\Y}(\gamma) \to \Y \times \PP^1$ is a double-point degeneration, and together with $\delta' = \delta_{\Y}(\gamma) \in \F'(\Y')$ it determines a relation on $\Cycle_{*,\F'}(\Y)$.  The claim is that this is precisely the relation we want.

By the base-change property, $\Y'_0 = \pi'^{-1}(\Y_0) = W_{\Y_0}(\res{\gamma}{\Y_0}$ and $\delta_{\Y_0}(\res{\gamma}{\Y_0}) = \res{\delta'}{\Y'_0}$. 
Similarly, $\Y'_1 = \pi'^{-1}(\Y_1) = \pi'^{-1}(A \cup_D B)$ is the union of smooth components $A' = W_{A}(\res{\gamma}{A})$ and $B' = W_B(\res{\gamma}{B})$ intersecting transversely along the smooth divisor $D' = W_D(\res{\gamma}{D})$ with all the $\delta$ pulled back from $\Y'$.

It remains to compare $\PP_{D'} = \PP_{D'}(\O_{D'} \oplus \N_{A'/D'})$ with $W_{\PP_D}$, as $D$-schemes.  By base-change, $W_{\PP_D} = D' \times_{D} \PP_D$. Thus it suffices to verify that
\[ \xymatrix{ \PP_{D'} = \PP_{D'}(\O_{D'} \oplus \N_{A'/D'})\ar[r] \ar[d] & D' \ar[d] \ar[r] & \Y' \ar[d] \\
\PP_{D} = \PP_D(\O_D \oplus \N_{A/D}) \ar[r]& D \ar[r] & \Y } \] is Cartesian.  This is true since formation of normal bundles is smooth local, i.e., $\N_{A'/D'}$ is the pullback of $\N_{A/D}$ to $D'$.
\end{proof}

\begin{corollary}\label{cor:opsFlag} The operation (iii) (i.e., $(\bG/\bB)_{-}$) is well-defined.
\end{corollary}
\begin{proof} It satisfies the hypothesis of the previous Lemma, with the cycle representative being simply $[\pi \circ p: (\bG/\bB)_{\P} \to Y, \P^\B]$.  Furthermore, $\LL$-linearity is geometrically clear.
\end{proof}


\section{Tori and Borels}\label{sec:bTbB}
\subsection{$\omega_{*,\bT}$ vs $\omega_{*,\bB}$} 
In topology, there is a homotopy equivalence $B\bB \sim B\bT$ since $\bB$ and $\bT$ differ by the (contractible) $\bN$. This motivates the following Lemma.

\begin{lemma}\label{lem:bTbB} Suppose $\bT \subset \bB$ is the inclusion of a maximal torus into a Borel.  Then, $\ind_{\bT}^{\bB}: \omega_{*,\bT}(X) \to \omega_{*,\bB}(X)$ is an isomorphism of $\LL$-modules, with inverse induced from a quotient map $\bB \to \bB/\bN = \bT$.
\end{lemma}
\begin{proof} 
   Fix a retraction
    \[  \bT \stackrel{i}{\longrightarrow} \bB \stackrel{p}{\longrightarrow} \bB/\bB_u \isom \bT \] which induces
    \[ \omega_{*,\bT}(X) \stackrel{\ind_i}{\longrightarrow} \omega_{*,\bB}(X) \stackrel{\ind_p}{\longrightarrow} \omega_{*,\bT}(X) \] 
    so that $\ind_i$ is injective. (Think ``direct sum of a sequence of line bundles'' and ``associated graded of a filtered vector bundle''.)  
    
    To remains to show that $\ind_i \ind_p = \id$; more concretely, if $\P \to Y$ is a $\bB$-torsor, we wish to give a cobordism (of $\bB$-torsors) between $\P \to Y$ and $\ind_i \ind_p \P \to Y$ (i.e., to ``take the extension parameter to zero'').  For this, it suffices to construct a family of group endomorphisms $f_t: \bB \to \bB$, $t \in \AA^1$, such that $f_1 = \id$, $f_0 = i \circ p$ is the retraction $\bB \to \bT$, and $f_t$ is an inner automorphism for $t \neq 0$.  By e.g., applying $f_t$ to cocycle representatives, this gives a $\bB$-torsor $\P' \to Y \times \AA^1$ whose restriction to $Y \times \{0\}$ is isomorphic to $\ind_i \ind_p \P$ and whose restriction to $Y \times \GG_m$ is isomorphic to $p_1^* \P$.  The second property allows us to extend $\P'$ to a $\bB$-torsor over all of $Y \times \PP^1$ by gluing on $\res{p_1^* \P}{Y \times \PP^1-\{0\}}$ along $Y \times \GG_m$.

    
Finally, there is a standard construction of $f_t$ by using the conjugation action of $\bT$: Take $\lambda: \GG_m \to \bT$ to be a strictly dominant coweight, and set $f_t(b) = \lambda(t) b \lambda(t)^{-1}$.  By the following standard Lemma, $f_t \to i \circ p$ as $t \to 0$. \qedhere 
\end{proof}

\begin{lemma} Suppose $\bT \subset \bG$ is a maximal torus, $\bB$ a Borel, and $\lambda: \GG_m \to \bT$ a strictly dominant coweight.  The induced action map  \[ a_{\lambda}: \GG_m \times \bB \to \bB \qquad a_{\lambda}(t,b) = \lambda(t) b \lambda(t)^{-1} \] extends uniquely to a map $a'_{\lambda}: \AA^1 \times \bB \to \bB$, and furthermore $\res{a'_{\lambda}}{0 \times \bB} = i \circ p$ is the retraction of $\bB$ onto $\bT$.
\end{lemma}
\begin{proof} See \cite[Prop.~8.4.5, Ex.~8.4.6(5)]{Springer}.  A representative case is the easy matrix computation for $\mathbf{SL_2}$: 
  \[ \lambda(t) \begin{pmatrix} a & b \\ 0 & c \end{pmatrix} \lambda(t)^{-1} = \begin{pmatrix} a & b t^{\langle \lambda, r_{+} \rangle} \\ 0 & c \end{pmatrix}. \qedhere \]
\end{proof}

\begin{remark} If $\bH$ entered into $\omega_{*,\bH}(X)$ only through (the $\AA^1$-homotopy type of) $B\bH$ then Lemma~\ref{lem:bTbB} would be easy (like in topology), and an analogous statement would hold more generally for the map to the reductive quotient $\bH \to \bH_{\text{red}}$ for any linear algebraic $\bH$.  The above proofs generalize, using the special form of the Levi decomposition, to show that this holds if $\bH$ is a parabolic subgroup of a reductive group.  We don't know if it holds in general.
\end{remark}

\subsection{Reminders on $\omega_{*,\bT}(X)$}
We want to emphasize the following Corollary of computations in \cite{LP}, parts of which are not stated in precisely the form we need there:
\begin{lemma}\label{lem:bt}\mbox{} \begin{enumerate}
  \item There is an isomorphism of $\LL$-modules
    \[ \omega_{*,\bT} \isom \bigoplus_{m_1,\ldots,m_r \geq 0} p_{m_1,\ldots,m_r} \cdot \LL \] where
    \[ p_{m_1,\ldots,m_r} = [\PP^{m_1} \times \cdots \PP^{m_r}, (\O_{\PP^{m_1}}(1), \cdots, \O_{\PP^{m_r}}(1))]. \]
  \item The pairing $\ip{-}{-}_\bT= \LL\ps{\oh{\bT}} \otimes \omega_{*,\bT} \to \LL$ is a perfect pairing of $\LL$-modules, in the sense that it induces $\LL\ps{\oh{\bT}} = (\omega_{*,\bT})^\dual$.
  \item The natural map
    \[ \gamma^X_{\bT}: \omega_{*,\bT} \otimes_{\LL} \omega_{*}(X)  \longrightarrow  \omega_{*,\bT}(X) \] is an isomorphism of $\LL$-modules.
  \item Suppose $x \in \omega_{*,\bT}(X)$.  Then, $x = 0$ iff $\ip{A}{x}_\bT = 0$ for all $A \in \LL\ps{ \oh{\bT} }$.
\end{enumerate}
\end{lemma}
\begin{proof}\begin{enumerate}
  \item This follows from the more explicit $\QQ$-basis given in \cite{LP}.  Let us indicate the necessary translation: Given $(m_1,\ldots,m_r)$ and a partition $\lambda'$ indexing a basis element of $\LL$, take the ``partition list'' $(\lambda, (m_1,\ldots,m_r))$ where $\lambda = \lambda' \cup (m_i)$.  Note that \cite{LP} shows that these elements generate integrally and are a basis (in particular, independent) rationally.
  \item  We can compute
    \begin{align*} \ip{\lambda_1^{m'_1} \cdots \lambda_r^{m'_r}}{p_{m_1, \ldots, m_r}}_\bT &= c_1\left(\O_{\PP^{m_1}}(1)\right)^{m'_1} \circ \cdots \circ c_1\left(\O_{\PP^{m_r}}(1)\right)^{m'_r} \cap [\PP^{m_1} \times \cdots \times \PP^{m_r}] \\ & = [\PP^{m_1-m'_1} \times \cdots \times \PP^{m_r-m'_r}] \end{align*}  So the matrix for the pairing is block triangular (vanishes if any $m'_i > m_i$), with ones on the diagonal (where all $m'_i=m_i$).
   \item Surjectivity was established in \cite{LP} along the way to proving Theorem~3 there.  Injectivity follows by a Chern-operation argument, which can be extracted from the part of \cite{LP} giving a basis for $\omega_{*,1^r}$.  Alternatively, (ii) implies (by ``coefficient solving'') that there exist unique $d_{n'_1,\ldots,n'_r}\in \LL\ps{\oh{\bT}}$ such that
     \[ \ip{d_{m'_1,\ldots,m'_r}}{p_{m_1,\ldots,m_r}}_\bT = \begin{cases}  1 & \text{if $m'_1=m_1$, \dots, $m'_r=m_r$} \\ 0 & \text{otherwise} \end{cases} \]  This gives rise to an inverse to $\gamma^X_\bT$, since
       \[ \ip{d_{m'_1\ldots,m'_r}}{\gamma^X_\bT(p_{m_1,\ldots,m_r} \otimes x)}_\bT = \ip{d_{m'_1,\ldots,m'_r}}{p_{m_1,\ldots,m_r}}_\bT x  \in \omega_*(X). \] 
     \item Clear from the above. \qedhere
\end{enumerate}
\end{proof}


\section{Proof of Rational Results}\label{sec:rational}
Our proof will follow the same general outline as the reduction of $\bGLr$ to $\GG_m^r$ in Lee-Pandharipande.  First, we show surjectivity of the map by considering (after generically finite base-change) sections of the flag bundle associated to $\P$, followed by ``setting the extensions parameters to zero.''  Then, we show injectivity by using suitable characteristic classes.

\subsection{Surjectivity in Theorem~\ref{thm:coinv}}
For surjectivity, we may as well work a bit better than rationally from the start.

\begin{prop}\label{prop:surj} The map $\ind_{\bT}^{\bG}: \omega_{*,\bT}(X)/W \to \omega_{*,\bG}(X)$ of Theorem~\ref{thm:coinv} becomes surjective after tensoring $-\otimes_{\ZZ} \ZZ\otg$.
\end{prop}
\begin{proof} By Lemma~\ref{lem:grNAK}, it suffices to prove that the cokernel of the induced map $ \sq{\ind_{\bT}^{\bG}}: \sq{\omega}_{*,\bT}(X)/W \longrightarrow \sq{\omega}_{*,\bG}(X) $ is killed by $\tau_\bG$.  (Recall $\sq{\omega} = \omega \otimes_{\LL} \ZZ$ in all variants.)

  Let $[\pi: Y \to X, \P]$ be a generator of $\sq{\omega}_{*,\bG}(X)$.  Now, let's stare at the diagram
  \[ \xymatrix{
  &  \P/\bB \ar@{=}[d] & \P/\bN \ar@{=}[d] & \P \subset p^* \P \ar@{=}[d]\\
  Z_i \ar@{^{(}->}[r] & \ar[d]^p (\bG/\bB)_{\P} & \P^{\bT} \ar[l] & \P^{\bB} \subset p^* \P \ar[l] \\ Y_i \ar[u] \ar[r]_{r_i} & Y \ar[d]^{\pi} & &\P\ar[ll] \\ & X
  }\]
  Here, $p: (\bG/\bB)_{\P} \to Y$ is the generalized flag bundle of $\P$; $\P^{\bB} \to (\bG/\bB)_{\P}$ is a $\bB$-bundle, in fact a $\bB$-reduction of $p^* \P$ which is the universal instance of reduction of the structure group of $\P$ to $\bB$; $\P^{\bT} \to (\bG/\bB)_{\P}$ is a $\bT$-bundle, the one obtained from $\P^{\bB}$ by ``setting the extension parameters to zero.''

  By Prop.~\ref{prop:BGsect}, there are $Z_1,\ldots,Z_k \subset (\bG/\bB)_{\P}$ which are generically finite of degree $d_i$, $\gcd d_i = \tau_\bG$, over $Y$.   Let $Y_i$ be a resolution of singularities of $Z_i$, so that $Y_i$ is smooth and the composite $r_i: Y_i \to Z_i \to Y$ is still generically finite of degree $d$.  
  By Prop.~\ref{prop:lm}, $[Y_i \to Y] = p_* [Y_i \to (\bG/\bB)_{\P}] = (\deg r_i) [Y] \in \sq{\omega}_*(Y)$.   Taking an appropriate linear combination, we may find $x \in \omega_{*}((\bG/\bB)_{\P})$ such that $p_* x = \tau_\bG [Y] \in \sq{\omega}_*(Y)$.  By Lemma~\ref{lem:bTbB}, $\ind_{\bT}^{\bB}(x \bullet \P^\bT) = x \bullet \P^{\bB}$; since $\P^{\bB}$ was a $\bB$-reduction of $p^* \P$, $\ind_{\bB}^{\bG}(x \bullet \P^{\bB}) = x \bullet p^* \P$.  Consequently,
  \begin{align*} \ind_{\bT}^{\bG} \pi_* p_*(x \bullet \P^\bT) &=  \pi_* p_*\left(\ind_{\bT}^{\bG}\left(x \bullet \P^{\bT}\right) \right) \\ &=\pi_* p_*\left(x \bullet p^* \P\right) = \pi_*\left((p_* x) \bullet \P\right) \\ &= \tau_{\bG} \pi_*([Y] \bullet \P) = \tau_{\bG}[\pi: Y \to X, \P] \in \sq{\omega}_{*,\bG}(X). \qedhere\end{align*}
\end{proof}

\subsection{Proof of injectivity in Theorem~\ref{thm:coinv}}


\begin{proof}
  Suppose $x \in \omega_{*,\bT}(X)_{\QQ}$ is such that $\Psi(x) = 0 \in \omega_{*,\bG}(X)_{\QQ}$. Since we're working rationally $(\omega_{*,\bT}(X)_{\QQ})^W = (\omega_{*,\bT}(X)_{\QQ})_W $ by averaging, so to prove injectivity it suffices, by Lemma~\ref{lem:bt}, to show that \[ \ip{A}{ \sum_{w \in W} w x }_\bT = \ip{\underbrace{\left(\sum_{w \in W} w \cdot A\right)}_{\in \LL_\QQ\ps{ \oh{T} }^W}}{x}_\bT = 0  \] for all $A \in \LL_\QQ\ps{ \oh{T} }$.   But, the indicated term lies in $\LL_\QQ\ps{ \oh{T} }^W$ and by naturality of the construction we obtain
  \[ \ip{\sum_{w \in W} w \cdot A}{x}_{\bT} = \ip{\sum_{w \in W} w \cdot A}{\Psi(x)}_\bG = 0 \] as desired. 
\end{proof}

\subsection{Proof of Theorem~\ref{thm:main}}
\begin{proof} \mbox{}
  \begin{enumerate} \item Consider the diagram
      \[ \xymatrix{ 
      \omega_{*}(X) \otimes_{\LL} (\omega_{*,\bT})/W \ar[r]^-{\gamma_\bT^X} \ar[d]_{\id \otimes \ind_\bT^\bG} & \omega_{*,\bT}(X)/W \ar[d]^{\ind_\bT^\bG} \\
      \omega_{*}(X) \otimes_{\LL} \omega_{*,\bG} \ar[r]_-{\gamma_\bG^X} & \omega_{*,\bG}(X) } \] By Theorem~\ref{thm:coinv} the vertical arrows become isomorphisms after tensoring with $\QQ$, so that in order to prove that the bottom horizontal arrow becomes an isomorphism it suffices to prove that the top horizontal arrow does so.  The top arrow is gotten by taking Weyl coinvariants of a map which is an isomorphism by Lemma~\ref{lem:bt}, and is thus an isomorphism.
    \item Consider the diagram
      \[ \xymatrix{ 
      (\omega_{*,\bT})/W \ar[d] \ar[r]^-{\vartheta_{\bT}/W} & \MU_*(B\bT(\CC))/W \ar[d] \\
      \omega_{*,\bG} \ar[r]_-{\vartheta_\bG} & \MU_*(B\bG(\CC)) 
      }\]
      Theorem~\ref{thm:coinv} tells us that the left vertical arrow becomes an isomorphism after tensoring with $\QQ$.   The right vertical arrow is rationally an isomorphism by a well-known transfer argument in topology.  The top horizontal arrow is obtained by taking Weyl coinvariants of a map which is an isomorphism by Lemma~\ref{lem:bt}.  It follows that the bottom horizontal arrow is an isomorphism.\qedhere
  \end{enumerate}
\end{proof}


\section{Better than rationally}\label{sec:better-than-Q}
\subsection{The retraction}
The argument used to prove surjectivity in Theorem~\ref{thm:coinv} yields the following more precise statement:
\begin{lemma}\label{lem:retract} Suppose $A \in \LL\ps{\oh{\bT}}$, $|A|=d = \dim\bG/\bB$, is such that $p_*(c_{\bT}(A) \cap [\bG/\bB]) = \tau_{\bG} [\pt] \in \sq{\omega}_0(\pt) \isom \CH_0(\pt)$.  Let $\phi: \omega_{*,\bG}(X) \to \omega_{*,\bG}(X)$ denote the composite
  \[ \xymatrix@C=3pc{
  \omega_{*,\bG}(X) \ar[r]^-{(\bG/\bB)_{-}} & \omega_{*+d, \bB}(X) \ar[r]^-{\c_\bT(A) \cap -} & \omega_{*, \bB}(X) \ar[r]^-{\ind_{\bB}^{\bG}} & \omega_{*, \bG}(X) } \] 
  Then, $\phi \equiv \tau_{\bG} \pmod{\mm}$ (i.e., $\phi(x) - \tau_{\bG} x \in \LL_{\geq 1} \omega_{*, \bG}(X)$).

 The same statements holds with $\omega_{*,\bG}$ replaced by $\omega^{Zar}_{*,\bG}$ throughout.
\end{lemma}
\begin{proof} Suppose $[\pi: Y \to X, \P] \in \omega_{*,\bG}(X)$.  Let $p: (\bG/\bB)_{\P} \to Y$ be the projection, and $\P^\bB$ the natural $\bB$-reduction of $p^* \P$.  Tracing through the definitions we find
  \begin{align*}
    \phi([\pi: Y \to X, \P]) &= \ind_{\bB}^{\bG} \pi_* p_*\left(\left(c_\bT(A, \P^B) \cap [(\bG/\bB)_\P]\right) \bullet \P^{\bB} \right) \\
    &= \pi_* p_*\ind_{\bB}^{\bG} \left(\left(c_\bT(A, \P^B) \cap [(\bG/\bB)_\P]\right) \bullet \P^{\bB} \right) \\
    &= \pi_* p_*\left(\left(c_\bT(A, \P^B) \cap [(\bG/\bB)_\P]\right) \bullet p^* \P \right) \\
    &= \pi_*\left(p_*\left(c_\bT(A, \P^B) \cap [(\bG/\bB)_\P]\right) \bullet \P \right) 
  \end{align*}  
  Consequently, it suffices to show that $ p_*\left(c_\bT(A, \P^B) \cap [(\bG/\bB)_\P]\right) = \tau_\bG [Y] \in \sq{\omega}_*(Y)$.   By Prop.~\ref{prop:lm}(ii), it suffices to prove the same formula in Chow theory.  It's enough to prove this for each connected component of $Y$, so we are asking for an equality of elements in $\CH_{0}(Y) \isom \ZZ$. Finally, we may check equality after base-change to an \'etale neighborhood trivializing $\P$ (or a closed point).  This reduces us to the case of $p: \bG/\bB \to \pt$, where the equality follows by hypothesis.
%
%
Nothing is changed by requiring $\P$ to be Zariski-locally trivial, so the same holds with $\omega^{Zar}_{*,\bG}$.
\end{proof}

Even though it seems like we've done very little, we can read off some consequences:
\begin{prop}[Theorem~\ref{thm:better-than-Q}(i), part of (iii)]\label{prop:retract} \mbox{}
  \begin{enumerate} \item The composite maps \[ \phi\otg: \omega_{*,\bG}(X)\otg \to \omega_{*,\bG}(X)\otg \qquad \text{and} \qquad \phi'\otg: \omega^{Zar}_{*,\bG}(X)\otg \to \omega^{Zar}_{*,\bG}(X)\otg \] of Lemma~\ref{lem:retract} are isomorphisms.
    \item The natural map $\omega^{Zar}_{*,\bG}(X)\otg  \to \omega_{*,\bG}(X)\otg$ is an isomorphism.
    \item $\omega_{*,\bG}(X)\otg$ is a direct summand of $\omega_{*,\bT}(X)\otg$.
    \item $\omega_{*,\bG}$ is a projective $\LL\otg$-module, and a torsion-free $\ZZ$-module.
  \end{enumerate}
\end{prop}
\begin{proof} By definition of $\tau_\bG$, there exists $A \in \LL\ps{\oh{\bT}}$ satisfying the hypothesis of Lemma~\ref{lem:retract}.  So, claim (i) follow by combining Lemma~\ref{lem:retract} and Lemma~\ref{lem:grNAK} applied to $\phi$.

Since the map of (i) is an isomorphism, the diagram of Lemma~\ref{lem:retract} displays $\omega_{*,\bG}(X)\otg$ (and $\omega^{Zar}_{*,\bG}(X)\otg$) as a retract of $\omega_{*, \bB}(X)\otg \isom \omega_{*,\bT}(X)\otg$ (Lemma~\ref{lem:bTbB}); this proves (iii).  Analogously, this exhibits the morphism $\omega_{*,\bG}(X)\otg \to \omega^{Zar}_{*,\bG}(X)\otg$ as a retract of the isomorphism $\omega^{Zar}_{*,\bB}(X)\otg = \omega_{*,\bB}(X)\otg$, proving (ii).

Finally, $\omega_{*,\bT}\otg$ is a free $\LL\otg$-module, in particular a torsion-free $\ZZ$-module, by Lemma~\ref{lem:bt}.  So any direct summand of it is a projective $\LL\otg$-module and a torsion-free $\ZZ$-module.
\end{proof} 

\begin{prop}[Theorem~\ref{thm:better-than-Q}(ii)]  The map
  \[ \gamma^X_{\bG}\otg: \omega_*(X)\otimes_\LL \omega_{*,\bG}\otg \longrightarrow \omega_{*,\bG}(X)\otg \] is an isomorphism of $\LL\otg$-modules.
\end{prop}
\begin{proof} Consider the diagram
  \[\xymatrix{
  \omega_*(X)\otimes_\LL \omega_{*,\bG}\otg \ar[r]^-{\gamma^X_{\bG}} \ar[d]&  \omega_{*,\bG}(X)\otg  \ar[d] \\
  \omega_*(X)\otimes_\LL \omega_{*,\bB}\otg \ar[r]^-{\gamma^X_{\bB}} \ar[d]&  \omega_{*,\bB}(X)\otg  \ar[d] \\
  \omega_*(X)\otimes_\LL \omega_{*,\bG}\otg \ar[r]^-{\gamma^X_{\bG}}&  \omega_{*,\bG}(X)\otg  } \]
  where the right-hand vertical maps are just those of Lemma~\ref{lem:retract} for $X$, and the left-hand vertical maps are those of Lemma~\ref{lem:retract} for $X = \pt$ tensored with $\omega_*(X)$.  By Prop.~\ref{prop:retract}, the diagram exhibits $\gamma^X_{\bG}$ as a retract of $\gamma^X_{\bB}$.  Combining Lemma~\ref{lem:bTbB} and Lemma~\ref{lem:bt}(ii) shows that $\gamma^X_{\bB}$, and consequently its retract $\gamma^X_{\bG}$, is an isomorphism.
\end{proof}

\subsection{Completing the proof}
In order to relate $\omega_{*,\bT}\otg$ and $\omega_{*,\bG}\otg$, we will use ``better-than-rationally'' characteristic classes and in particular the ``operation'' $c_{\bG}$.
%

\begin{proof}[Completion of proof of Theorem~\ref{thm:better-than-Q} and Prop.~\ref{prop:direct-summand}]
  It remains to identify the $\LL\otg$-module dual of $\omega_{*,\bG}\otg$.  By Prop.~\ref{prop:surj},\[  \omega_{*,\bT}\otg/W \twoheadrightarrow \omega_{*,\bG}\otg \] and consequently
  \[ \left(\omega_{*,\bG}\otg\right)^\dual \hookrightarrow \left(\omega_{*,\bT}\otg/W \right)^\dual = \left(\left(\omega_{*,\bT}\otg\right)^\dual\right)^W \] (where $M^\dual = \Hom_{\LL\otg}(M, \LL\otg)$).

  Using $\c_\bT$ and $\c_\bG$ (Section~\ref{sec:operations}) for $X = \pt$ --- or more accurately the pairings $\varepsilon \c_\bT = \ip{-}{-}_\bT$ and $\varepsilon \c_\bG = \ip{-}{-}_\bG$ derived from them --- and the compatibility between them, we obtain the commutative diagram
  \[ \xymatrix{
  \LL\otg\ps{\oh{\bT}}^W \ar[r]^{\varepsilon \c_{\bG}} \ar[dr]_{\varepsilon \c_{\bT}} & \left(\omega_{*,\bG}\otg\right)^\dual  \ar@{^{(}->}[d] \\ &\left(\left(\omega_{*,\bT}\otg\right)^\dual\right)^W  } \]  By Lemma~\ref{lem:bt}(iv), the diagonal arrow is an isomorphism.  So, the vertical arrow is a surjection and thus an isomorphism.  This completes the proof of Theorem~\ref{thm:better-than-Q}.

  Since $\omega_{*,\bG}\otg$ is a projective $\LL\otg$-module (Prop.~\ref{prop:retract}(iv)), it injects into its double dual.  We thus obtain a commutative diagram
  \[ \xymatrix{
  \omega_{*,\bT}\otg \ar@{->>}[r] \ar[d] & \omega_{*,\bG} \ar@{^{(}->}[d]\\
  \left( \LL\otg\ps{\oh{\bT}}^W  \right)^\dual \ar[r]_-{\sim} & {{\omega_{*,\bG}}^\dual}^\dual }
  \] 

  It follows that $\omega_{*,\bG}$ is isomorphic to the image of $\omega_{*,\bT}\otg$ in $\left( \LL\otg\ps{\oh{\bT}}^W  \right)^\dual$.  But this is precisely the quotient of $\omega_{*,\bT}\otg$ obtained by modding out by the submodule of $x \in \omega_{*,\bT}\otg$ for which $\ip{A}{x}_{\bT} = 0$ for all $A \in \LL\otg\ps{\oh{\bT}}^W$.  This proves Prop.~\ref{prop:direct-summand}(i).

  To prove Prop.~\ref{prop:direct-summand}(ii), we combine the surjectivity result of Prop.~\ref{prop:surj}, the rational results of Theorem~\ref{thm:coinv}, and the assertion that $\omega_{*,\bG}\otg$ is torsion-free  (Prop.~\ref{prop:retract}(iv)) to conclude that there is a commutative diagram
  \[ \xymatrix{ \left(\omega_{*,\bT}\otg/W\right)/\text{torsion} \ar@{^{(}->}[r] \ar@{->>}[d] & (\omega_{*,\bT})_\QQ/W \ar[d]^{\sim} \\
  \omega_{*,\bG}\otg \ar@{^{(}->}[r] & (\omega_{*,\bG})_\QQ } \] which implies that $\omega_{*,\bG}\otg \isom \left(\omega_{*,\bT}\otg/W\right)/\text{torsion}$.
\end{proof}

\begin{remark}\label{rmk:depends} There is an alternative way to carry out the previous proof that removes the dependence on $\S$\ref{ss:charBG-2}, and in particular on a better-than-rational $\c_\bG$, by leveraging the rational statements of Theorem~\ref{thm:coinv} and the torsion-freeness of Prop.~\ref{prop:retract}(iv).\footnote{If the reader makes it to the proof of Corr.~\ref{corr:charBG2}, they will find that in fact this sort of argument, using torsion-freeness to leverage a rational argument, was actually used there.  So in a sense this amounts to shifting this up a level in our constructions (doing it after passing to cobordism, rather than before), making things a tiny bit easier at the cost of no longer getting the results of $\S$\ref{ss:charBG-2} as pleasant side-effects.  This is analogous to how Lemma~\ref{lem:retract} is a cobordism analog of the $\ZZ\otg$ form of Lemma~\ref{lem:equaliz}.}
  
Alternate proof:  By Theorem~\ref{thm:coinv} and torsion-freeness of $\omega_{*,\bG}\otg$, we obtain that the kernel of $\omega_{*,\bT}\otg \twoheadrightarrow \omega_{*,\bG}\otg$ is precisely the set of $x \in \omega_{*,\bG}\otg$ s.t. $\ip{A}{x} = 0 \in \LL_\QQ$ for all $A \in \LL_\QQ\ps{\oh{\bT}}^W$.  Using torsion freeness, and the vanishing of high enough filtration elements on any fixed $x$, we see that this is equivalent to $\ip{A}{x}=0$ for all $A \in \LL\otg\ps{\oh{\bT}}^W$.  This proves Prop.~\ref{prop:direct-summand}(i).  

By the rational results, the kernel of $\omega_{*,\bT}\otg/W \twoheadrightarrow \omega_{*,\bG}\otg$ is torsion. Since $\LL\otg$ is torsion-free, the surjection induces an isomorphism on $\LL\otg$-duals.  Combining with Lemma~\ref{lem:bt}, this proves Theorem~\ref{thm:better-than-Q}(iii) by Lemma~\ref{lem:bt}.
\end{remark}

\appendix

\bibliography{note}

\end{document}

%% file: macros.tex
\usepackage{textcmds}  
\usepackage{amsfonts,  amssymb, amsmath, amsthm, amsopn, bm, latexsym, bbm}
\usepackage{enumitem}
\usepackage[ps,matrix,curve,frame,arrow,rotate,line]{xy}

\usepackage{fullpage}

\usepackage{mathrsfs}
\usepackage[mathcal]{eucal}

\usepackage[backref,colorlinks=true,linkcolor=blue,citecolor=blue,urlcolor=blue,citebordercolor={0 0 1},urlbordercolor={0 0 1},linkbordercolor={0 0 1}]{hyperref} 
\usepackage[backrefs,alphabetic,abbrev]{amsrefs} 

\let\oldfootnotemark\footnotemark
\let\oldfootnotetext\footnotetext
\let\oldfootnote\footnote
\renewcommand\footnote[1]{\addtocounter{footnote}{1}\hypertarget{fnbackref.\arabic{footnote}}{}\addtocounter{footnote}{-1}\oldfootnote{#1\fnbackref}}
\renewcommand\footnotemark{\addtocounter{footnote}{1}\hypertarget{fnbackref.\arabic{footnote}}{}\addtocounter{footnote}{-1}\oldfootnotemark}
\renewcommand\footnotetext[1]{\oldfootnotetext{#1\fnbackref}}
\newcommand{\fnbackref}{\hyperlink{fnbackref.\arabic{footnote}}{\footnotesize$\uparrow$}}

\makeatletter
\newcommand{\sss}{\@startsection{subsubsection}{2}{0pt}{-3ex
plus -1ex minus -0.2ex}{-2mm plus -0pt minus
-2pt}{\normalfont\bfseries}} \makeatother

\let\oldmarginpar\marginpar
\renewcommand\marginpar[1]{\mbox{}\oldmarginpar[\raggedleft\hspace{0pt}\textcolor{Red}{\footnotesize #1}]%
{\raggedright\hspace{0pt}\textcolor{Red}{\footnotesize #1}}}

\theoremstyle{plain}
\newtheorem{theorem}{Theorem}[subsection]
\newtheorem{rconjecture}[theorem]{Conjecture}
\newtheorem{prop}[theorem]{Proposition}
\newtheorem{lemma}[theorem]{Lemma}
\newtheorem{corollary}[theorem]{Corollary}

\theoremstyle{definition}
\newtheorem{defn}[theorem]{Definition}
\newtheorem{constr}[theorem]{Construction}

\newtheorem{remark}[theorem]{Remark}

\newtheorem{example}[theorem]{Example}

\newtheorem{na}[theorem]{}

	{\begin{Sbox}\begin{minipage}{\linewidth}\begin{rconjecture}}%
	{\end{rconjecture}\end{minipage}\end{Sbox}\par\vspace{2pt}\noindent\fbox{\vbox{\vspace{.25pt}\TheSbox\vspace{.25pt}}}}

	{\begin{Sbox}\begin{minipage}{\linewidth}\begin{theorem}}%
	{\end{theorem}\end{minipage}\end{Sbox}\par\vspace{2pt}\noindent\fbox{\vbox{\vspace{.25pt}\TheSbox\vspace{.25pt}}}}

\newtheorem{rprob}{Problem}

	{\begin{Sbox}\begin{minipage}{\linewidth}\begin{rprob}}%
	{\end{rprob}\end{minipage}\end{Sbox}\par\vspace{2pt}\noindent\fbox{\vbox{\vspace{.25pt}\TheSbox\vspace{.25pt}}}}

\newcommand{\isom}{\simeq}           

\newcommand{\eqdef}{\overset{\text{def}}{=}}

\newcommand{\dual}{{_\vee}}

\newcommand{\F}{\mathscr{F}}

\renewcommand{\L}{\mathscr{L}}

\newcommand{\Y}{\mathscr{Y}}

\newcommand{\B}{\mathcal{B}}

\renewcommand{\O}{\mathcal{O}}

\renewcommand{\AA}{\mathbb{A}}
\newcommand{\CC}{\mathbb{C}}
\newcommand{\GG}{\mathbb{G}}

\newcommand{\LL}{\mathbb{L}}

\newcommand{\PP}{\mathbb{P}}
\newcommand{\QQ}{\mathbb{Q}}

\newcommand{\ZZ}{\mathbb{Z}}

\newcommand{\mm}{\mathfrak{m}}

\newcommand{\Card}[1]{\##1}

\newcommand{\ol}[1]{\overline{#1}}

\newcommand{\oh}[1]{\widehat{#1}}
\newcommand{\sq}[1]{\widetilde{#1}}

\newcommand{\ip}[2]{\left\langle#1,#2\right\rangle}
\newcommand{\res}[2]{\left. #1 \right|_{#2}}

\newcommand{\adjunct}[2]{\xymatrix@1{ #1 \ar@<.7ex>[r] & \ar@<.7ex>[l] #2 }}

\newcommand{\ssetr}[2]{\xymatrix@1{ #1 \ar[r] & \ar@<.7ex>[l] #2 \ar@<-.7ex>[l] \ar@<.7ex>[r] \ar@<-.7ex>[r] & \cdots \ar@<1.4ex>[l] \ar@<-1.4ex>[l] \ar[l] }}

\newcommand{\ssetl}[2]{\xymatrix@1{ \cdots \ar@<1.4ex>[r] \ar@<-1.4ex>[r] \ar[r] & #2 \ar@<.7ex>[r] \ar@<-.7ex>[r] \ar@<.7ex>[l] \ar@<-.7ex>[l] & #1 \ar[l]}}

\newcommand{\ssetrr}[3]{\xymatrix@1{ #1 \ar[r] & #2 \ar@<.7ex>[l] \ar@<-.7ex>[l] \ar@<.7ex>[r] \ar@<-.7ex>[r] & #3 \ar@<1.4ex>[l] \ar@<-1.4ex>[l] \ar[l]  \ar@<1.4ex>[r] \ar@<-1.4ex>[r] \ar[r] & \cdots\ar@<.7ex>[l] \ar@<-.7ex>[l] \ar@<2.1ex>[l] \ar@<-2.1ex>[l]}}

\newcommand{\ssetll}[3]{\xymatrix@1{ \cdots\ar@<.7ex>[r] \ar@<-.7ex>[r] \ar@<2.1ex>[r] \ar@<-2.1ex>[r] & #3 \ar@<1.4ex>[r] \ar@<-1.4ex>[r] \ar[r]  \ar@<1.4ex>[l] \ar@<-1.4ex>[l] \ar[l] & #2 \ar@<.7ex>[r] \ar@<-.7ex>[r] \ar@<.7ex>[l] \ar@<-.7ex>[l] & #1 \ar[l]}}

\newcommand{\GL}{\operatorname{GL}}

\DeclareMathOperator{\id}{id}

\DeclareMathOperator{\im}{im}

\DeclareMathOperator{\End}{End}

\DeclareMathOperator{\Hom}{Hom}

\DeclareMathOperator{\Sym}{Sym}

\DeclareMathOperator{\Spec}{Spec}

\newcommand{\pt}{\ensuremath{\text{pt}}}